\documentclass[12pt,table]{amsart}

\usepackage{amsmath,amsfonts,amssymb,amsthm,amstext,hyperref,verbatim,lmodern,ytableau,textcomp,young,youngtab,tikz,subfiles,tikz-cd,bbold}
\usepackage{xcolor}
\usepackage{enumerate,verbatim,mathtools,graphicx,pgf}
\usepackage{mathrsfs}

\usepackage[mathscr]{euscript}
\usepackage[utf8]{inputenc}

\usepackage{calligra}
\usepackage{setspace,fancyhdr}
\theoremstyle{plain}
          \newtheorem{theorem}{Theorem}[section]
          \newtheorem{lemma}[theorem]{Lemma}
           \newtheorem{conj}[theorem]{Conjecture}
           \newtheorem{proposition}[theorem]{Proposition}
        \theoremstyle{definition}
          \newtheorem{example}[theorem]{Example}
          \newtheorem{remark}[theorem]{Remark}
          \newtheorem{defn}[theorem]{Definition}

\numberwithin{equation}{section}
\usepackage[square,sort&compress,comma,numbers]{natbib}
\usepackage{hyperref}

\definecolor{amber}{rgb}{1.0, 0.75, 0.0}
\definecolor{ao}{rgb}{0.0, 0.50, 0.0}
\usepackage[a4paper,top=3cm,bottom=2cm,left=3cm,right=3cm,marginparwidth=1.75cm]{geometry}

\usepackage{amsmath}
\usepackage{graphicx}
\usepackage[colorinlistoftodos]{todonotes}

\title{Correlations in the continuous multispecies 
 TASEP on a ring}

\author{Nimisha Pahuja}
    \address{Department of Mathematics, Indian Institute of Science, Bangalore - 560012}
    \email{nimishap@iisc.ac.in}
\author{Surjadipta De Sarkar}
    \address{Department of Mathematics, Indian Institute of Science, Bangalore - 560012}
    \email{surjadiptade@iisc.ac.in}

\DeclareMathOperator{\iden}{\langle 1^n \rangle}
\begin{document}

\begin{abstract}
In this paper, we prove a conjecture proposed by Aas and Linusson regarding the two-point correlations of adjacent particles in a continuous multispecies TASEP on a ring (AIHPD, 2018). We use the theory of multiline queues, as devised by Ferrari and Martin (AOP, 2008) to interpret the conjecture in terms of the placements of numbers in triangular arrays. Further, we use projections to calculate the correlations in the continuous multispecies TASEP using a distribution on these placements.
\end{abstract}

\keywords{multispecies, correlation, TASEP, continuous, lumping}
%
%
\maketitle


\section{Introduction}\label{sec:Intro}

Multispecies exclusion processes and their combinatorial properties have been a popular topic of investigation in recent times~\cite{Intro_AAV, Intro_angel, uchiyama2005correlation, arita2009spectrum}. One property which is of great interest
is the correlation of two or more particles in the stationary distribution of the process~\cite{Intro_AALP_Phase,aaslinusson,ayyerlinusson}. The
aim of this paper is to prove a conjecture of Aas and Linusson~\cite{aaslinusson} on correlations
of two adjacent points in a multispecies totally asymmetric exclusion process (TASEP)
on a continuous ring.

A lot of attention has been given to various properties of the multispecies TASEP in recent years. Connections between TASEPs and various mathematical structures, including affine Weyl groups~\cite{Intro_Lam, Intro_AALP_Weyl}, Macdonald polynomials~\cite{Intro_CMW_Macdonald}, Schubert polynomials~\cite{Intro_Kim_Williams}, and multiline queues~\cite{ferrarimartin} have been explored extensively. Multispecies TASEP on a ring has been studied in \cite{Intro_AL_inhomo, Intro_Aas_Jonas, Intro_Cantini}. Ayyer and Linusson~\cite{ayyerlinusson} studied multispecies TASEP on a ring, where they proved conjectures by Lam~\cite{Intro_Lam} on random reduced words in an affine Weyl group and gave results on two-point and three-point correlations.

One of the first instances where the continuous multispecies TASEP on a ring was mentioned is by Aas and Linusson~\cite{aaslinusson}. They studied a distribution that should be a certain infinite limit of the stationary distribution of multispecies TASEP on a ring. They also conjectured~\cite[Conjecture 4.2]{aaslinusson} a formula for correlations $c_{i,j}$ which is given by the probability that the two particles, labelled $i$ and $j$ are next to each other with $i$ on the left of $j$ in the limit distribution. We prove this conjecture first for the case $i>j$  (Theorem~\ref{thm:i>j}) in Section~\ref{sec:i>j} and then for the case $i<j$ (Theorem~\ref{thm:i<j}) in Section~\ref{sec:i<j}.  The technique we use is similar to and inspired by the work of Ayyer and Linusson~\cite{ayyerlinusson}, where they study correlations in multispecies TASEP on a ring with a finite number of sites.

To carry out the analysis, we use the theory of \textit{multiline process} that Ferrari and Martin described in~\cite{ferrarimartin}. The multiline process is defined on structures known as \textit{multiline queues} or MLQs. This process can be projected to the multispecies TASEP using a procedure known as \textit{lumping of chains} (see~\cite[Lemma 2.5]{LPW_book}). This lets us study the stationary distribution of the multiline process to infer results for the stationary distribution of the multispecies TASEP. It is defined using an algorithm known as \textit{bully path projection}, which projects a multiline queue to a word. See Section~\ref{sec:pre} for the precise definitions. 

We study the two-point correlations in a continuous TASEP with type $\iden = \underbrace{(1,\ldots,1)}_n$. In this case, there are $n$ particles, each with a distinct label from the set $[n]=\{1,\ldots,n\}$. Let $c_{i,j}(n)$ denote the probability that two particles, labelled $i$ and $j$, lie adjacent on a ring in the limiting distribution for the continuous multispecies TASEP with type $\iden$, with $i$ followed by $j$.  Aas and Linusson gave an explicit conjecture (\cite[Conjecture 4.2]{aaslinusson})  for calculating $c_{i,j}(n)$. We prove their conjecture in this paper separately for the two cases. 

\begin{theorem}\label{thm:i>j}
 For $n \geq 2$ and $i>j$, we have the following two-point correlations:
\[c_{i,j}(n)=\begin{cases}
\frac{n}{\binom{n+j}{2}}-\frac{n}{\binom{n+i}{2}}, & \text{ if }j < i < n, \\
\frac{n(j+1)}{\binom{n+j}{2}}-\frac{n(j-1)}{\binom{n+j-1}{2}}-\frac{n}{\binom{2n}{2}}, & \text{ if } j < i = n. 
\end{cases}\]
\end{theorem}
\begin{theorem}\label{thm:i<j}
For $n \geq 2$ and $i<j$, we have the following two-point correlations:
\[c_{i,j}(n)=\begin{cases}
\frac{n}{\binom{n+j}{2}}, & \text{ if }i+1 < j \leq n, \\
\frac{n}{\binom{n+j}{2}}+\frac{ni}{\binom{n+i}{2}}, & \text{ if }i+1 = j \leq n.
\end{cases}\]
 \end{theorem}
We begin Section~\ref{sec:pre} by establishing key definitions, followed by a description of the bully path projection and the encoding of multiline queues as Young tableaux.  This sets up the combinatorial framework for analysing multiline queue configurations. These methods are first applied to the base case of continuous two-species TASEP in Section~\ref{sec:twospecies}, deriving initial correlation expressions that motivate the general case. Theorem~\ref{thm:i>j} is proved in Section~\ref{sec:i>j} using the techniques developed in Section~\ref{sec:pre}. In Section~\ref{sec:i<j}, we introduce preliminary lemmas for the case $i<j$, which lead to a complete proof of Theorem~\ref{thm:i<j} in Section~\ref{sec:progress}, thereby resolving the conjecture by Aas and Linusson. Finally, Section~\ref{sec:betaproof} offers an alternative approach to solving Theorem~\ref{thm:i<j} using first principles. 

\section{Preliminaries}\label{sec:pre} 

A multispecies TASEP  is a stochastic process on a graph. We first define multispecies TASEP on a ring before studying the continuous multispecies TASEP on a ring.

\subsection{Multispecies TASEP}\label{sec:mTASEP}
A multispecies TASEP is a continuous-time Markov process that can be defined on a ring with $L$ sites. For an integer tuple $\textbf{m}=(m_1, \ldots, m_n)$, a multispecies TASEP of type $\textbf{m}$ has $m_1+\cdots+m_n$ sites on the ring occupied with particles. Each particle is assigned a label from the set $[n]$, and there are exactly $m_{\ell}$ particles with the label $\ell$. The unoccupied sites are treated as particles with the highest label $n+1$. The dynamics of the process are as follows: Each particle carries an exponential clock that rings with rate $1$. When the clock rings, the particle tries to jump to the next site in the clockwise direction. Let this particle be labelled $i$. The jump is successful only if the following site has a label greater than $i$ or is unoccupied (label $n+1$). In that case, the two particles exchange positions. The states of the multispecies TASEP are words of length $L$ with the letter $\ell$ occurring $m_{\ell}$ times for all $\ell \in [n]$, and $n+1$ occurring $L-\sum_{\ell}m_{\ell}$ times.  

Now we define the \textit{multiline process}, which can be projected onto the multispecies TASEP through a process called \textit{lumping}. A multiline process is a Markov process defined on a graph that consists of a collection of disjoint cycle graphs, all of the same length. Each graph is represented as a \textit{row}, numbered from top to bottom. Each row contains the same number of sites, which may or may not be occupied by a particle. For an $n$-tuple $\textbf{m}=(m_1,\ldots, m_n)$ with $m_{\ell} \geq 0$ and $L\geq m_1+\cdots+m_n$, a \textit{multiline queue} of type $\textbf{m}$ is formed by stacking $n$ rows on top of each other, each with $L$ sites. In the $i^{th}$ row from the top, $C_i:=m_1+\cdots+ m_i$ of the sites are occupied. See Figure~\ref{fig:MLQ} for an example of a multiline queue.
\begin{figure}[ht]
   \centering
   \includegraphics[width=.6\textwidth]{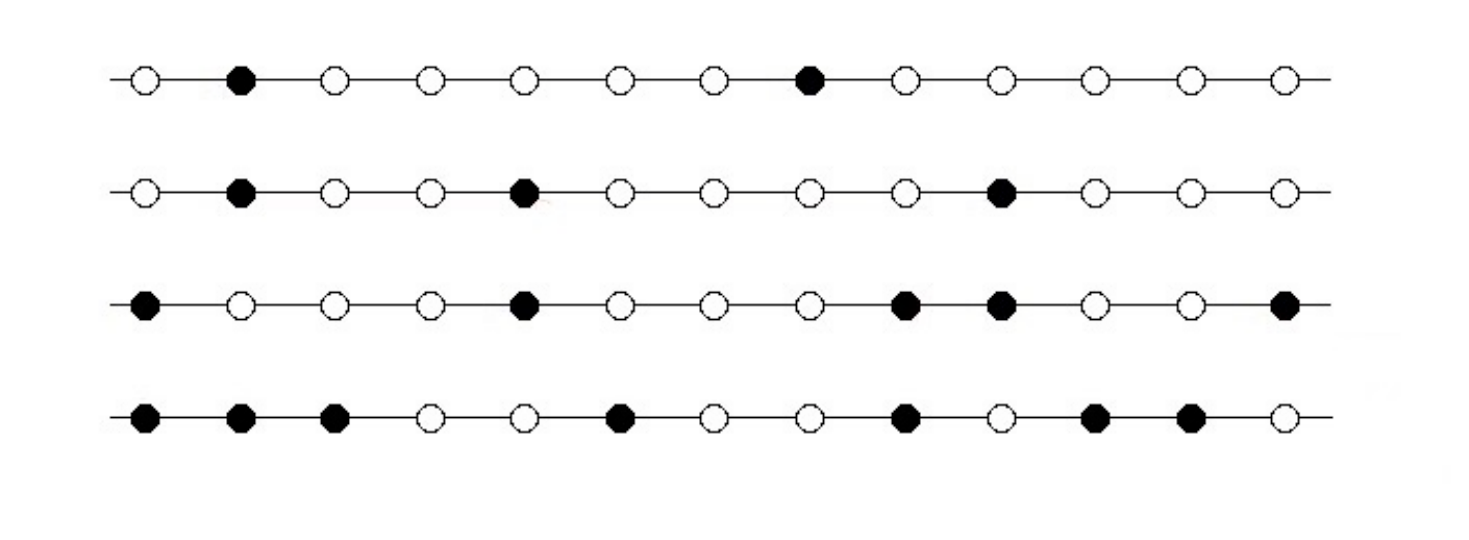}
    \caption{A multiline queue of type $(2,1,2,2)$ on $13$ sites.}
    \label{fig:MLQ}
\end{figure}

The dynamics of the multiline process are described in detail in~\cite{ferrarimartin} via transitions on multiline queues of a fixed type. 
A multiline queue of type $\textbf{m}$ can be projected to a word by an algorithm known as the \textit{bully path procedure}, which we define recursively as follows:
\begin{enumerate}
    \item \label{item:Step1} Let $Q$ be a multiline queue of type $\textbf{m}$. We construct \textit{bully paths} that contain exactly one particle from each row. Start with the first row in $Q$. Choose a particle in the first row, and trace its path down to the second row. The path then moves rightwards along the second row until it runs into another particle. If no particle is found by moving rightwards, the path wraps around the end of the second row and continues searching from the leftmost site of the second row until it finds one. This case is known as bullying via \textit{wrapping}. This process is repeated in the third row, where the path again moves downwards and then rightwards until another particle is encountered. This sequence continues all the way down to the last row. Each particle encountered by this bully path is labelled $``1"$. We similarly construct the bully paths starting from other particles in the first row, and labelling unlabelled particles hit by the path. It turns out that the order in which these paths are constructed, starting from the particles in the first row, does not matter for the final labelling. At the end of this step, we have a total of $m_1$ bully paths. That is, $m_1$ particles of the last row are labelled $``1"$. See Figure \ref{fig:bpe1} for the construction of bully paths to the multiline queue in Figure~\ref{fig:MLQ}. Here, the bully paths starting from each row are constructed in a left-to-right order.

\begin{figure}[ht]
   \centering
    \includegraphics[width=.6\textwidth]{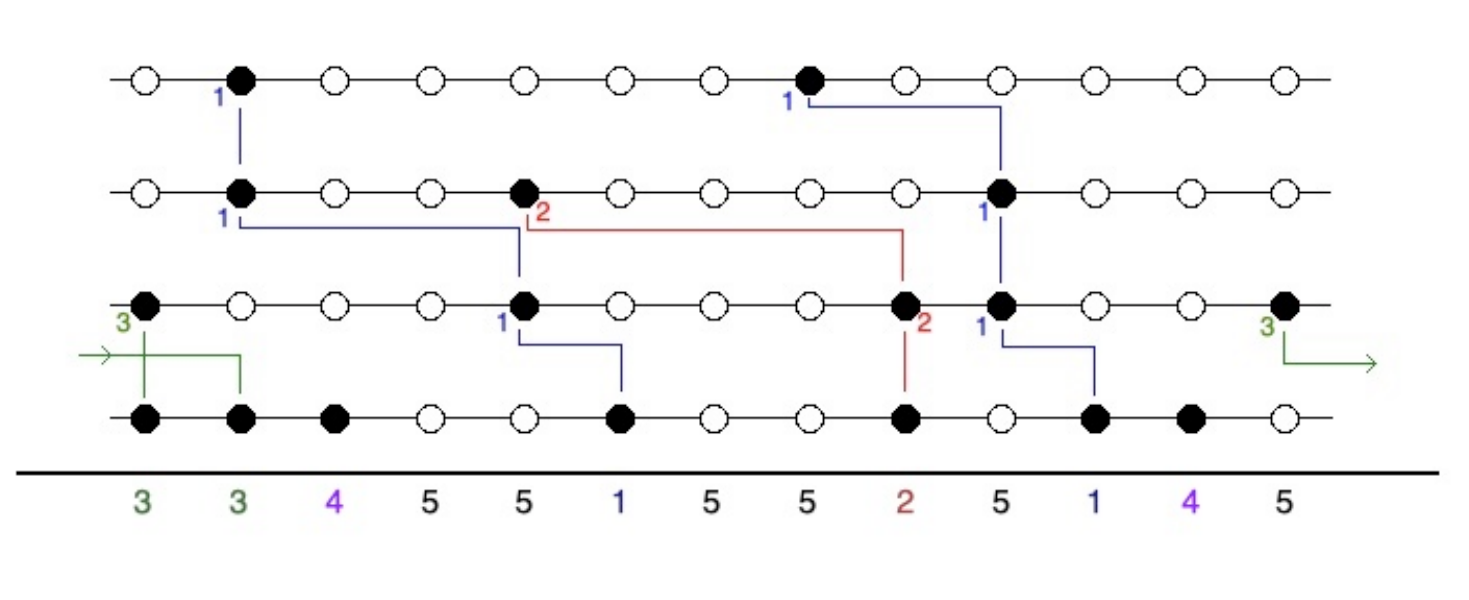}
    \caption{Bully path projection on the multiline queue from Figure~\ref{fig:MLQ}. The bully paths starting in the first, second, and third rows are shown in colors blue, red, and green, respectively.}
    \label{fig:bpe1}
\end{figure}
    
    \item Next, we construct bully paths starting with the unlabelled particles in the second row by repeating the same process from Step~(\ref{item:Step1}). We label all the particles in these paths at the end as $``2"$, and there are $m_2$ paths in total. We then repeat these steps for all the subsequent rows. Finally, we label all the particles that remain unlabelled in the last row as $``n"$ and all the unoccupied sites as $``n+1"$.
    
    \item Hence for each $\ell$, $m_{\ell}$ particles in the last row are labelled as $``\ell"$. Let $\omega$ denote the word formed by the labels in the last row. Then, $\omega$ is a configuration of the multispecies TASEP of type $\textbf{m}$. Let $\mathcal{B}$ denote this projection map. Then, $\omega$ is known as the \textit{projected word} of $Q$ and we write it as $\omega=\mathcal{B}(Q)$. The projected word in Figure ~\ref{fig:bpe1} is 3345515525145.
\end{enumerate}
The connection between the stationary distribution of the multiline process and that of the multispecies TASEP is established by the following theorem by Ferrari and Martin~\cite{ferrarimartin}. 

\begin{theorem}\cite[Theorem 4.1]{ferrarimartin}
\label{thm:fm2}
The process on the last row of the multiline process of type $\textbf{m}$ is the same as the multispecies TASEP of type $\textbf{m}$.
\end{theorem}

\subsection{Continuous multispecies TASEP}
 Fix $\textbf{m}=(m_1, \ldots, m_n)$ and define $C_i=m_1+\cdots+m_i$ for all $i \in [n]$. The continuous multispecies TASEP can be viewed as a formal limit of the stationary distribution of the multispecies TASEP  on a ring with $L$ sites. First, we consider a multispecies TASEP of type $\textbf{m}$ on a ring with $L$ sites. Let $\Pi_n^L$ denote the stationary distribution of this TASEP. As we let $L$ approach infinity while keeping the tuple $\textbf{m}$ constant (which implies that the number of unoccupied sites tends to infinity), the ring is scaled to the continuous interval $[0,1)$. The limit of the stationary distribution $\Pi_n^L$ then yields a distribution $\Pi_n$ of labelled particles on a continuous ring. It is important to note that $\Pi_n$ is not yet shown as the stationary distribution of any Markov process.

Similar to the multiline queues described in \cite{ferrarimartin}, we can define continuous multiline queues of a given type. For $\textbf{m}=(m_1, \ldots, m_n)$, let $C_i = m_1+\cdots +m_i$. A \textit{continuous multiline queue} of type $\textbf{m}$ consists of $n$ copies of the continuous ring $[0,1)$ stacked on top of each other, with $C_i$ particles in the $i^{th}$ row from the top.

The location of each particle is considered to be a real number within the continuous interval $[0,1)$. In the distribution we will consider, the horizontal position of each particle will almost surely be distinct.
\begin{example}\label{exmp:MLQ}
See Figure \ref{fig:contMLQ} for an example of a continuous multiline queue of type $(1,3,1,2)$. The rows have $1,4,5$ and $7$ particles respectively. Note that there is no particle directly above or below any other particle. 
\end{example}
 
\begin{figure}[ht]
   \centering
    \includegraphics[width=.6\textwidth]{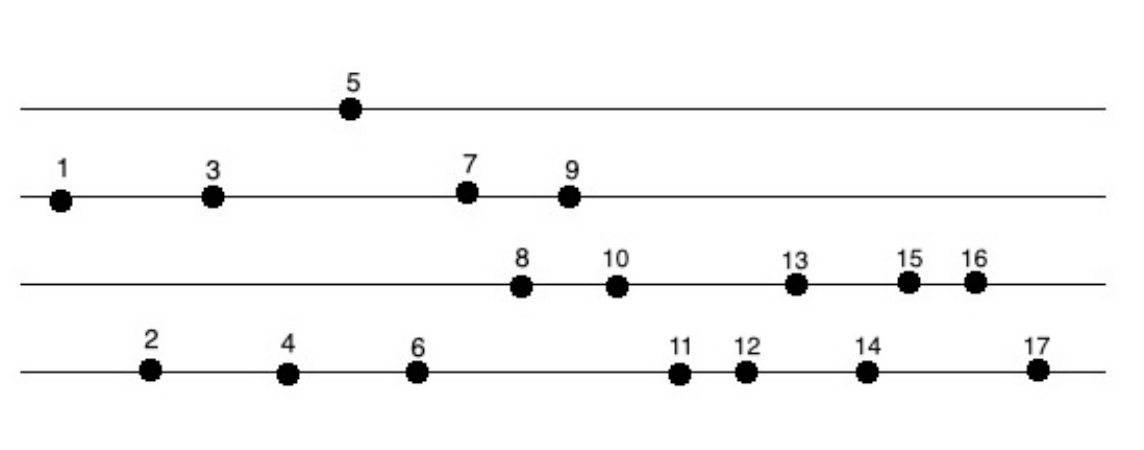}
    \caption{A continuous multiline queue of type $(1,3,1,2)$}
    \label{fig:contMLQ}
\end{figure}

Consider the labels of the particles in Figure~\ref{fig:contMLQ}. The labels are assigned in the order of the horizontal positions of the particles, moving from left to right. We now refer to an integer representation of a continuous multiline queue, which was also used by Aas and Linusson~\cite{aaslinusson}.

\begin{defn}
 Let $\textbf{m}=(m_1,\ldots,m_n)$ be an $n$-tuple, with $C_i=m_1 + \cdots + m_i$, and let $N=\sum_{i=1}^n C_i$. A \textit{placement} of a continuous multiline queue of type $\textbf{m}$ is represented as a triangular array $(Q_{i,j})$ with distinct integers from the set $[N]$ such that the integer $Q_{i,j}$ denotes the relative horizontal position of the $j^{th}$ particle in the $i^{th}$ row of the continuous multiline queue as seen from left to right. 
 \end{defn}
 \begin{remark}
For our purposes, it suffices to know the relative positions of particles within the rows. Therefore, a continuous multiline queue will be represented by its placement, and we will use these two terms interchangeably. The number of different placements of type $\textbf{m}$ is given by $\binom{N}{C_1\,,\ldots,\,C_n}$.
\end{remark}
 \begin{example}\label{exmp:placement}
 The placement of the continuous multiline queue in Figure~\ref{fig:contMLQ} is given by
\[Q=\begin{array}{ccccccc}
     5 & &&&&&\\
     1 & 3 & 7 & 9 \\
     8 & 10 & 13 & 15 & 16 \\
     2 & 4 & 6 & 11 & 12 & 14 & 17.
\end{array}
\]

Each row of the array is formed by arranging the order of the particles from the corresponding row of the continuous multiline queue in ascending order based on their horizontal positions.
\end{example}

The \textit{bully path projection} is a map from the set of continuous multiline queues of type $\textbf{m}$ to words of type $\textbf{m}$. Given a continuous multiline queue $Q$, the algorithm recursively maps $Q$ to a word $\omega$ as follows:

\begin{enumerate}
    \item\label{item:cbp1} Consider the placement of $Q$. Choose an integer $k_1$ in the first row. Look for the smallest element larger than $k_1$ in the second row and mark it as $k_2$, if it exists. If $k_1$ is larger than all the available integers in the second row, mark the smallest available integer in the second row as $k_2$. This is known as \textit{wrapping} from the first row to the second row. We say that $k_1$ ``bullies" $k_2$ and write it as $k_1 \rightarrow k_2$; or $k_1 \xrightarrow{W} k_2$ in the case of wrapping.

    \item\label{item:cbp2} Mark an integer in the third row as $k_3$ by finding the smallest integer larger than $k_2$. If no such integer exists, wrap around and label the smallest integer as $k_3$. Continue this process iteratively. The sequence $k_1,k_2, \ldots,k_n$ thus obtained is called a \textit{bully path} starting at $k_1$, and these integers are now unavailable for further bullying. Similarly, construct the bully paths starting from other integers in the first row, and labelling available integers hit by the path. Label the endpoints of all such paths as $1$. There are $m_1$ such paths, called \textit{type $1$} bully paths. For example, $5\rightarrow 7\rightarrow 8 \rightarrow 11$ is a type $1$ bully path in Figure \ref{fig:bpe}. The order of constructing the bully paths, starting from the first row, does not affect the final labelling.
    
    \item Next, construct bully paths starting with the available integers in the second row by following steps~(\ref{item:cbp1}) and (\ref{item:cbp2}). Label the ends of all such paths with $2$, resulting in $m_2$ bully paths of type $2$. Repeat these steps sequentially for all the other rows. The bully paths of type $n$ are just the integers in the last row that remain after the construction of all type $(n-1)$ bully paths. 
    \begin{figure}[ht]
    \centering
   \includegraphics[width=.6\textwidth]{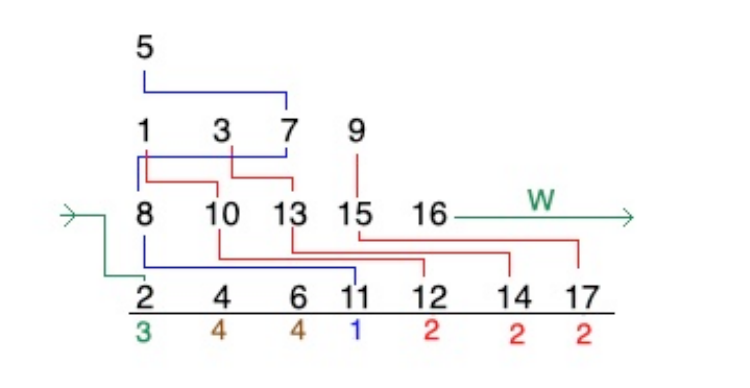}
    \caption{Bully path procedure on a multiline queue of type (1,3,1,2). The types $1,2$ and $3$ bully paths are shown in colour blue, red, and green, respectively.}
    \label{fig:bpe}
    \end{figure}
    \item Therefore, there are  $m_\ell$ bully paths of type $\ell$ for each $\ell$. Let $\omega$ denote the word formed by the labels in the last row. Then, $\omega$ is a configuration for the continuous multispecies TASEP of type $\textbf{m}$. Let $\mathcal{B}$ denote this projection map. Then, $\omega = \mathcal{B}(Q)$ is known as the \textit{projected word} of $Q$. The projected word for the continuous multiline queue in Example~\ref{exmp:placement} is 3441222; see Figure~\ref{fig:bpe}. In this example, the bully paths originating from each row are constructed in increasing order. 
\end{enumerate}

The distribution of words of type $\textbf{m}$ is the same as the distribution of the last row of continuous multiline queues of type $\textbf{m}$, which can be obtained by taking the limit of the distribution of the last row for discrete multiline queues. 

By Theorem~\ref{thm:fm2}, $\Pi_n^L$ denotes the distribution of the last row of the discrete multiline process of type $\textbf{m}$ with $L$ sites in each row. Therefore, as shown in \cite{aaslinusson}, the continuous process on the last row is obtained by taking the limit as $L \rightarrow \infty$. Consequently, $\Pi_n$ gives the distribution of the labels in the last row for a uniformly chosen continuous multiline queue. Thus, to study the correlations of two adjacent particles with labels $i$ and $j$ in the continuous multispecies TASEP, it is enough to count the number of placements that project to the words with $i$ and $j$ as adjacent particles. Next, we define an operator on the space of all continuous multiline queues of a fixed type.
 
\begin{defn}
  Let $Q$ be a continuous multiline queue of type $\textbf{m}$. Define an operator  $\mathcal{S}$ such that if $Q' = \mathcal{S}(Q)$, then $Q'$ is obtained from $Q$ by adding $1$ modulo $N$ to each element of $Q$, and rearranging any row in increasing order if necessary. We refer to $\mathcal{S}$ the \textit{shift operator} and to $Q'$ as the \textit{shifted} continuous multiline queue.
  \begin{example}\label{exmp:shift}
 Let $Q$ be the continuous multiline queue from Example \ref{exmp:placement}, then $Q'$ is given by 
 \[Q'=\begin{array}{ccccccc}
     6 & &&&&&\\
     2 & 4 & 8 & 10 \\
     9 & 11 & 14 & 16 & 17 \\
     1 & 3 & 5 & 7 & 12 & 13 & 15.
\end{array}
\]
\end{example}
\end{defn}
The following lemma establishes the relationship between the projected word of a multiline queue and that of its shifted continuous multiline queue.

\begin{lemma}\label{lem:shift}
Let $Q$ be a continuous multiline queue of type $\textbf{m}$ with $N=\sum_{i=1}^n C_i$ particles. Shifting $Q$ rotates the projected word by one position to the right when the largest element $N$ is in the last row of the placement, while preserving the word otherwise. Specifically, if $Q'= \mathcal{S}(Q)$, $\omega=\mathcal{B}(Q)$, and $\omega'=\mathcal{B}(Q')$, then $\omega$ and $\omega'$ are related as follows:
\begin{enumerate}
\item if $N$ is not in the last row, then $\omega' = \omega$.
\item if $N$ is in the last row, then $\omega'$ is obtained by rotating $\omega$ one position to the right.
\end{enumerate}
\end{lemma}
\begin{proof}
Let $N$ be in the $r^{th}$ row of the placement of $Q$. First, let $r<n$. $Q'$ is obtained from $Q$ by adding $1$ to every integer less than $N$ and by replacing $N$ with $1$. By the increasing property of the rows, the $r^{th}$ row is now rotated by one position to the right. If $N$ lies on a bully path that starts in some row above $r$ in $Q$, then all the bully paths remain the same. This is true because $N$, the largest integer in $Q$, wraps around and bullies the first element available to it in $(r+1)^{st}$ row. In $Q'$, $1$ being the smallest integer bullies the first available element, which is exactly the translation of the element bullied by $N$ in $Q$. The remaining bully paths are the same since the inequalities among all other elements do not change. 
On the other hand, if a type $r$ bully path in $Q$ begins at $N$ such that if $s_1 < s_2 < \cdots < s_{m_r}$ are all the elements of the $r+1^{st}$ row, that are hit by type $r$ bully paths in $Q$, then the elements hit by the type $r$ bully paths in $Q'$ will be $s_1+1 < s_2+1 < \cdots < s_{m_r}+1$. The specific bullying relations, however, depend on the order in which the type $r$ paths are constructed. Consequently, the projected word remains the same.

Finally, when $r=n$, that is, when $N$ is in the last row, adding $1$ modulo $N$ to each integer results in the last row rotating one position to the right. The bully paths remain unchanged; thus, the projected word derived from the labels of the particles in the last row is also rotated one position to the right.
\end{proof}

Let $\iden=(1,\ldots,1)$ be an $n$-tuple. Let $\hat{Q}$ be a continuous multiline queue of type $\iden$ and $\eta=\mathcal{B}(\hat{Q})$ be the projected word of $\hat{Q}$. Define $c_{i,j}(n)=\mathbb{P}\{\eta_a=i, \eta_{a+1}=j; a \in [n]\}$, where $a+1$ is defined modulo $n$. 
To make the analysis of the continuous multispecies TASEP of type $\iden$ easy, we use a classical property known as the \textit{projection principle}, which states that particles of two consecutive types cannot be distinguished by particles of other types. This is a key observation in \cite[Section 1]{Intro_AAMP}. Thus,
identifying two consecutive labels $k$ and $k+1$ defines a natural lumping from the multispecies TASEP with $n$ species onto the multispecies TASEP with $(n-1)$ species. Therefore, to obtain $c_{i,j}(n)$ for $i>j$, it is enough to find the probability that $4$ is followed by a $2$ in the projected word of a random five-species continuous multiline queue with type $\textbf{m}_{i,j}=(j-1,1,i-j-1,1,n-i).$ 
  
We can further simplify the task by lumping many continuous multiline queues onto a three-species system. Given a continuous multispecies TASEP with type $\iden$, consider its lumping to the continuous multispecies TASEP of type $\textbf{m}_{s,t}=(s,t,n-s-t)$ where $s+t>j>s$ and $i>s+t$. Thus, a particle with label $j$ becomes a $2$, and that with label $i$ becomes a $3$ whenever $t,n-s-t>0$. To compute the correlation $c_{i,j}(n)$ for $i>j$, we need to look at the correlation of $3$ followed by $2$ in the projected words of continuous multiline queues of type $\textbf{m}_{s,t}$. Similarly, to formulate $c_{i,j}(n)$ for $i<j$, we need to look at the correlation of $2$ followed by $3$ in the projected words of continuous multiline queues with type $\textbf{m}_{s,t}$.

 For $1 \leq i,j \leq n$, and $a \in [n]$, let $E^a_{i,j}(n)=\mathbb{P}\{\eta_a=i, \eta_{a+1}=j \}$ where $a+1$ is defined modulo $n$. Thus,
\begin{equation}\label{cor}
c_{i,j}=\sum_{a=1}^{n} E^a_{i,j}.
\end{equation}
Consider a continuous multiline queue $Q$ of type $\textbf{m}_{s,t}$. Let the projected word of $Q$ be $\omega$. We define
\begin{align}
\Theta_a^<(s,t)=   \mathbb{P}\{\omega_a=2, \omega_{a+1}=3\}, \\
\text{and }\Theta_a^>(s,t)=   \mathbb{P}\{\omega_a=3, \omega_{a+1}=2\}     
\end{align}
 for $a \in [n]$. By the projection principle, we have 
\begin{align}
    \Theta_a^<(s,t) &=\sum_{j=s+t+1}^{n} \sum_{i=s+1}^{s+t}E^a_{i,j},\\
    \text { and }\Theta_a^>(s,t) &=
    \sum_{i=s+t+1}^{n} \sum_{j=s+1}^{s+t}  E^a_{i,j}.
\end{align}
Let $\mathscr{T}^<$ (respectively $\mathscr{T}^>$) denote that probability that $2$ and $3$ are adjacent. Then, it can be represented as the sum of $\Theta_a^<$ (respectively $\Theta_a^>$) over the index $a$. We have,
\begin{align}\label{tau1}
    \mathscr{T}^<(s,t) &  = \sum_{a=1}^{n} \sum_{j=s+t+1}^{n} \sum_{i=s+1}^{s+t} E^a_{i,j}(n)& \notag \\
      & =  \sum_{j=s+t+1}^{n} \sum_{i=s+1}^{s+t} c_{i,j}(n).
\end{align}
Equation \eqref{tau1} follows from \eqref{cor}. Similarly,
\begin{align}
    \mathscr{T}^>(s,t)& =  \sum_{i=s+t+1}^{n} \sum_{j=s+1}^{s+t} c_{i,j}(n).
\end{align}
 For $i<j$, the principle of inclusion-exclusion then gives us 
\begin{equation}\label{PIE1}
    c_{i,j}(n)= \mathscr{T}^<(i-1,j-i)-\mathscr{T}^<(i,j-i-1)-\mathscr{T}^<(i-1,j-i+1)+\mathscr{T}^<(i,j-i),
\end{equation}
and for $i>j$, we have
\begin{equation}\label{PIE2}
    c_{i,j}(n)= \mathscr{T}^>(j-1,i-j)-\mathscr{T}^>(j,i-j-1)-\mathscr{T}^>(j-1,i-j+1)+\mathscr{T}^>(j,i-j).
\end{equation}
For $a \in [n]$, we let
\begin{align}\label{Tst<}
T_a^<(s,t)&=\mathbb{P}\{\omega_a=2, \omega_{a+1}=3, Q_{3,n}=N\},\\
\text{ and } T_a^>(s,t)&=\mathbb{P}\{\omega_a=3, \omega_{a+1}=2, Q_{3,n}=N\},\label{Tst>}
\end{align}
where $N=n+2s+t$ is the number of integers in the placement of a continuous multiline queue of type $\textbf{m}_{s,t}$. Then, the following lemma holds.
\begin{lemma}\label{lem:rot} $\mathscr{T}^<(s,t)=(n+2s+t)\,T_a^<(s,t)$ for any $a \in [n]$.
\end{lemma}
\begin{proof}
Let $Q$ be a continuous multiline queue of type $\textbf{m}_{s,t}$. Applying the shift operator $\mathcal{S}$ sequentially $N=n+2s+t$ times to $Q$ generates all the rotations of the projected word. By Lemma~\ref{lem:shift}, in exactly $n$ out of these $N$ shifts, the projected word rotates one unit to the right, and in the remaining shifts, the projected word remains the same. For a fixed $a \in [n]$, any continuous multiline queue that contributes to $\mathscr{T}^<(s,t)$ can be obtained as a rotation of a continuous multiline queue for which
\begin{enumerate}[(i)]
    \item the projected word has $2$ and $3$ in positions $a$ and $a+1\mod n$ respectively, and
    \item $N$  is in the last row.
\end{enumerate}
Note that if a word has a $2$ followed by a $3$ in $p$ separate positions, it occurs as a rotation of $p$ different words with $\omega_a=2, \omega_{a+1}=3$. Hence, the result.   
\end{proof}
We also have $\mathscr{T}^>(s,t)=(n+2s+t)T_a^>(s,t)$ for any $a \in [n]$. Further, Lemma~\ref{lem:rot} holds for $a=1$ in particular. Henceforth, we will write $T_1^<$ (respectively $T_1^>$) as $T^<$ (respectively $T^>$) for simplicity.  Then, for $i<j$, \eqref{PIE1} gives 
 \begin{align}\label{PIE3}
   c_{i,j}(n) =& (n+i+j-2)\,T^<(i-1,j-i)-(n+i+j-1)\,T^<({i,j-i-1}) \notag\\
   &-(n+i+j-1)\,T^<({i-1,j-i+1}) +(n+i+j)\,T^<({i,j-i}),
\end{align}
and for $i>j$, \eqref{PIE2} gives 
\begin{align}
\label{PIE4}
   c_{i,j}(n) =& (n+i+j-2)\,T^>({j-1,i-j})-(n+i+j-1)\,T^>({j,i-j-1})\notag \\
   &-(n+i+j-1)\,T^>({j-1,i-j+1}) +(n+i+j)\,T^>({j,i-j})
\end{align}
respectively.

\subsection{Standard Young tableaux}\label{sec:syt}
 Consider a partition $\lambda$= $(\lambda_1,\lambda_2,\ldots, \lambda_\ell)$ of a positive integer $N$, where $\lambda_1 \geq \lambda_2 \geq \cdots  \geq \lambda_\ell > 0$ and $N=\lambda_1+\lambda_2 + \cdots +  \lambda_\ell$. We use the notation $\lambda \vdash N$ or $N=|\lambda|$ to indicate that $\lambda$ is a partition of $N$. The \textit{Young diagram} of a \textit{shape} (or partition) $\lambda$ is an arrangement of boxes in $\ell$ left-aligned rows, with $\lambda_i$ boxes in the $i^{th}$ row. The \textit{hook} of a box $x$ in the Young diagram is defined as the set of boxes that are directly to the right or directly below $x$, including $x$ itself. The \textit{hook length} of a box $x$, denoted by $h_x$, is the number of boxes in its hook. Below is an example of a Young diagram of shape $(5,3,1)$ with the hook lengths for each box indicated within the box.

\begin{center}
    \young(75421,421,1)
\end{center}
 A \textit{standard Young tableau} or an \textit{SYT} of shape $\lambda$ is a filling of a Young diagram of shape $\lambda$ with the numbers $1,2,\ldots, N$, where $N = |\lambda|$, such that the numbers increase strictly down a column and across a row. The following is an example of a standard Young tableau of shape $(5,3,1)$. 
\begin{center}
   \young(12569,378,4)
\end{center}
The number of standard Young tableaux of a given shape $\lambda$ can be counted using the following result, which is known as the \textit{hook-length formula}.

\begin{theorem}\cite{frame1954hook}\label{thm:hlf} Let $\lambda \vdash n$ be an integer partition. The number of standard Young tableaux of a shape $\lambda$ is given by
\begin{equation}\label{hlf}
 f_\lambda=\frac{N!}{\Pi_{x \in \lambda} h_x},  \end{equation}
 where the product is over all the boxes in the Young diagram of $\lambda$ and $h_x$ is the hook length of a box $x$.
\end{theorem}

 \begin{example}
For $\lambda=(a,b)$, the number of standard Young tableaux of shape $\lambda$ is given by 
     \begin{equation}\label{2tab}
        f_{(a,b)}=
         \frac{(a+b)!(a-b+1)}{(a+1)!b!}=\frac{a-b+1}{a+1}\binom{a+b}{a}.
\end{equation}
Similarly, for $\mu=(a,b,c)$, we have 
\begin{align}\label{3tab}
f_{(a,b,c)}&=\frac{(a+b+c)!(a-c+2)(a-b+1)(b-c+1)}{(a+2)!(b+1)!c!} \notag \\
&=\frac{(a-c+2)(a-b+1)(b-c+1)}{(a+2)(a+1)(b+1)}\binom{a+b+c}{a,b,c}.
\end{align}
\end{example}
\begin{remark}
    \label{rem:hookrec}
    It is straightforward to verify from \eqref{2tab}, that $f_{(a,b)}$ satisfies an interesting recurrence relation given by
    \begin{equation*}
        f_{(a,b)} = f_{(a-1,b)} +f_{(a,b-1)}.
    \end{equation*}
\end{remark}
\begin{lemma}\label{lem:syt}
    Let $\textbf{m}=(m_1,\ldots,m_n)$, and let $C_i=m_1+\cdots +m_i$. The set of continuous multiline queues of type $\textbf{m}$ that have no wrapping under the bully path projection is in bijection with standard Young tableaux of shape $\lambda$ where $\lambda_i=C_{n-i+1}=m_1+\cdots+m_{n-i+1}$.
    \end{lemma}
    
    \begin{proof}
        Consider $Q$, a continuous multiline queue of type $\textbf{m}$. The largest integer in $Q$ is given by the sum $N=C_1+\cdots+C_n$. For $1 \leq i \leq n, 1 \leq j \leq C_i$, let $Q_{i,j}$ be distinct integers from $1$ to $N$. Then, we have 
    \[Q=\begin{array}{cccccccc}
     Q_{1,1}  &\ldots & Q_{1,C_1}\\
     Q_{2,1} & \ldots &\ldots & Q_{2,C_2}\\
     \vdots\\
     Q_{n,1} & \ldots & \ldots & \ldots &Q_{n,C_n}.
     \end{array}\]

     If we assume that there is no wrapping from the first row to the second row in $Q$, then it follows that $Q_{1,C_1-j}<Q_{2,C_2-j}$ for all $0 \leq j \leq C_1-1$. Similarly, since there is no wrapping from the second row to the third row in $Q$, $Q_{2,C_2-j}<Q_{3,C_3-j}$ for all $0 \leq j \leq C_2-1$, and so on. This, along with the increasing property of the rows, gives us

\[
\begin{array}{ccccccccccccc}
     & & & & Q_{1,1}& < &\cdots & < & Q_{1,C_1-j} & < & \cdots & < & Q_{1,C_1}\\
     & & & & \wedge & & & & \wedge & & \cdots & & \wedge\\
     & &Q_{2,1} & < & \cdots & < &  \cdots & < & Q_{2,C_2-j} & < &\cdots & < & Q_{2,C_2}\\
     & & \wedge & &  &  &  & & \wedge & &  & & \wedge\\
     &&&&&&&&&&&&\vdots\\
     Q_{n,1} & < & \cdots& < & \cdots & < & \cdots & < &Q_{n,C_n-j} & <& \cdots & < & Q_{n,C_n}.
\end{array}\]
 Hence, $Q$ is in bijection with a standard Young tableau $\mathscr{Y}$ of shape $(C_n,\ldots,C_1)$ and the bijection is given by $\mathscr{Y}_{i,j}=N-Q_{n-i+1,C_{(n-i+1)}-j+1}+1$.
 Thus, the number of continuous multiline queues of type $\textbf{m}$ with no wrapping is given by $f_{(C_n,\ldots,C_1)}$. 
\end{proof}

\begin{defn}
    Given two partitions $\lambda,\mu$ such that $\mu \subseteq \lambda$ (containment order, that is, $\mu_i \leq \lambda_i$ for all $i$), the \textit{skew shape} $\lambda/\mu$ is a Young diagram that is obtained by subtracting the Young diagram of shape $\mu$ from that of $\lambda$. See Figure~\ref{fig:skewshape}.
\end{defn}

\begin{figure}[ht]
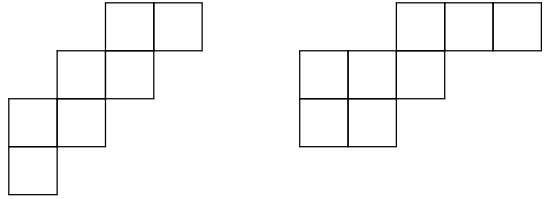

    \centering
    \ydiagram{2+2,1+2,2,1} \hspace{0.5in}\ydiagram{2+3,3,2}
    \caption{Skew shapes $(4,3,2,1)/(2,1)$ and $(5,3,2)/(2)$ respectively}
    \label{fig:skewshape}
\end{figure}

\begin{defn}
    A \textit{standard Young tableau of a skew shape} $\lambda/\mu$ is a filling of the skew shape by positive integers that are strictly increasing in rows and columns.
\end{defn}

\begin{example} Following are a few of the many standard Young tableaux of the skew shape $(4,3,2,1)/(2,1)$.
\begin{center}
\ytableausetup{notabloids}
\begin{ytableau}
\none & \none & 1 & 4 \\
\none & 2 & 3 \\
5 & 6 \\
7
\end{ytableau}
\hspace{0.5in}
\ytableausetup{notabloids}
\begin{ytableau}
\none & \none & 1 & 2 \\
\none & 5 & 7 \\
3 & 6 \\
4
\end{ytableau}
\hspace{0.5in}
\ytableausetup{notabloids}
\begin{ytableau}
\none & \none & 2 & 3 \\
\none & 1 & 7 \\
4 & 6 \\
5
\end{ytableau}
\end{center}
\end{example}

Let $f_{(\lambda/\mu)}$ denote the number of SYTs of a skew shape $\lambda/\mu$. This number is counted by the following formula. 
\begin{theorem}\cite[Corollary 7.16.3]{stanley2}
Let $|\lambda/\mu| = n $ be the number of boxes in the skew shape $\lambda/\mu$ that has $\ell$ parts. Then,
\begin{equation}\label{frob}
    f_{(\lambda/\mu)} = n! \text{det}\Big( \frac{1}{(\lambda_i-\mu_j-i+j)!}\Big) _{i,j=1}^\ell.
\end{equation}
\end{theorem}
Note that we take $0!=1$ and $k!=0$ for any $k<0$ as a convention. 

\begin{example}
    Let $\lambda/ \mu = (6,4)/(3)$. Then, $n=|\lambda/\mu| = 7$ and $\ell =2$. We have,
    
\[f_{(\lambda/\mu)} = 7! {\begin{vmatrix}
 \frac{1}{(6-3-1+1)!} & \frac{1}{(6-0-1+2)!}\\
 \frac{1}{(4-3-2+1)!} & \frac{1}{(4-0-2+2)!}
 \end{vmatrix}} = 34.\]
\end{example}

\subsection{The continuous two-species  TASEP}\label{sec:twospecies}
 In this section, we study the continuous TASEP with two species of particles. Due to the simplicity of its structure, the stationary distribution of the continuous two-species TASEP can be calculated completely. This analysis lets us demonstrate the technique of encoding the placement of a continuous multiline queue as a standard Young tableau for a simpler case. We use similar encodings frequently in later sections. 
 
 
 We use continuous multiline queues of type $(s,t)$ to find certain two-point correlations in continuous two-species TASEP. Let $\Omega_{s,t}$ denote the set of continuous multiline queues of type $(s,t)$, where each $Q \in \Omega_{s,t}$ consists of $s$ integers in the first row and $s+t$ integers in the second.
For $i \in \{1,2\}$, let $\delta_i(s,t)$ be the number of continuous multiline queues in $\Omega_{s,t}$ with the largest integer $2s+t$ in the second row, such that the projected word has $i$ in the first position. 

The multispecies TASEP satisfies rotational symmetry (see~\cite[Proposition 2.1]{ayyerlinusson}). By Lemma~\ref{lem:shift}, the stationary probability of a word $\omega$ is the same as that of the word $\omega'$, which is obtained by rotating $\omega$ one position to the right. Consequently, any particle occupying the first position when the largest integer is in the second row has a uniform distribution. Given that there are $s$ occurrences of label $1$ and $t$ occurrences of label $2$, the stationary probability that the first site is occupied by $1$ (respectively, $2$) is given by $\frac{s}{s+t}$ (respectively $\frac{t}{s+t}$). Furthermore, there are $\binom{2s+t-1}{s}$ continuous multiline queues in $\Omega_{s,t}$ with the integer $2s+t$ occurring in the second row. We have
\begin{equation}\label{del1}
\delta_1(s,t) = \frac{s}{s+t}\binom{2s+t-1}{s},
\end{equation}
\begin{equation}\label{del2}
\delta_2(s,t) = \frac{t}{s+t}\binom{2s+t-1}{s}.
\end{equation}
\begin{remark}\label{rem:rotsym}
 Note that the number of placements in $\Omega_{s,t}$ with $2s+t$ in the second row such that the projected word has $i$ in the $k^{th}$ position is the same for any $k \in [s+t]$ due to rotational symmetry. 
 \end{remark}

Similarly, let $\delta_{i,j}(s,t)$ count the number of continuous multiline queues of type $(s,t)$ that have the largest integer $2s+t$ in the second row such that the projected word has $i$ in the first position and $j$ in the second position. Based on the definitions of these quantities and using Remark~\ref{rem:rotsym}, we have the following system of independent equations for fixed $s$ and $t$:
\begin{align}
    \delta_{1,1} + \delta_{1,2}= \delta_{1},\label{del1dot}\\
    \delta_{1,1} + \delta_{2,1}= \delta_{1},\label{deldot1}\\
    \delta_{1,2} + \delta_{2,2}= \delta_{2}.\label{del2dot}
\end{align}
Therefore, finding $\delta_{i,j}$ for any $i,j \in \{1,2\}$ solves the system of equations. In particular, let $i=j=2$. Consider an arbitrary continuous multiline queue $Q$ with the following configuration such that $\pi = \mathcal{B}(Q)$.
\[
\begin{array}{cccccccc}
     &a_1 & a_2 & \ldots & a_s\\
    &b_1 & b_2 & \ldots &\ldots & b_{s+t-1} & 2s+t\\
     \hline
     \pi=&2 & 2 & \ldots & \ldots & \ldots & \ldots
\end{array}.\]

Since $\pi_1=\pi_2=2$, neither $b_1$ nor $b_2$ are bullied by any $a_k$ in $Q$. Therefore, $a_1> b_2$. By the increasing property of the rows, we also have $b_2>b_1$, which forces $b_1=1$ and $b_2=2$.
Furthermore, there is no wrapping from the first row to the second row, since $\pi_1 \neq 1$. This implies that the integers in $Q$ greater than $b_1$ and $b_2$ satisfy the following inequalities:  
\[\begin{array}{ccccccccccccc}
      & & & & a_{1} &  < & \cdots & < & a_{s-1} & < &a_s\\
      & & & & \wedge& &  \cdots & & \wedge & & \wedge\\
     b_{3} & < & \cdots & < &  b_{t+1}  & < &\cdots & < & b_{s+t-1} & < &2s+t.
\end{array}\]
Such configurations are in bijection with standard Young tableaux of shape $(s+t-2,s)$ by Lemma~\ref{lem:syt}. Therefore, using Equation \eqref{2tab}, we get 
\begin{equation}\label{del22}
\delta_{2,2}(s,t)=f_{(s+t-2,s)}=\frac{t-1}{2s+t-1}\binom{2s+t-1}{s} .
\end{equation}
Using \eqref{del1}-\eqref{del22}, we can solve for all the remaining $\delta_{i,j}$ as follows:
\begin{align}
    \delta_{1,2}(s,t) & = \delta_{2,1}(s,t) = \frac{t+1}{2s+t-1}\binom{2s+t-1}{s-1},\\
    & \delta_{1,1}(s,t) = 2\binom{2s+t-2}{s-2}.
\end{align}

\section[Correlations for the case i > j]{Correlations for the case $i>j$}
\label{sec:i>j}

In this section, we prove Theorem~\ref{thm:i>j} using the tools developed in Section~\ref{sec:pre}. 
Let $\textbf{m}_{s,t}=(s,t,n-s-t)$ and let $\mathscr{S}_{s,t}$ be the set of all continuous multiline queues $Q$ of type $\textbf{m}_{s,t}$ that satisfy $\omega_1=3$, $\omega_2=2$, and $Q_{3,n}=N=n+2s+t$, where $\omega= \mathcal{B}(Q)$. The set $\mathscr{S}_{s,t}$ consists of the continuous multiline queues which have the following configuration, where  $a_i,b_i$ and $c_i$ are distinct integers from the set $[N]$.
\[
\begin{array}{cccccccc}
     &a_1 & a_2 & \ldots &\ldots & a_s\\
     &b_1 & b_2 & \ldots & b_k &\ldots & b_{s+t}\\
     &c_1 & c_2 & \ldots & \ldots & \ldots & \ldots &c_n=N.\\
     \hline
     \omega=&3 & 2 & \ldots & \ldots & \ldots & \ldots & \ldots
\end{array}\]
\begin{lemma}\label{lem:cond}
Let $Q$ be a continuous multiline queue of type $\textbf{m}_{s,t}$. Then, $Q \in \mathscr{S}_{s,t}$ if and only if the following conditions hold:
\begin{enumerate}
    \item \label{item:1} If $a_1 \rightarrow b_k$, then $b_k>c_2$,
    \item \label{item:2} $c_1=1$ and $b_1=2$,
    \item \label{item:3} there is no wrapping from either of the two rows.
\end{enumerate} 
\end{lemma}
\begin{proof}
Let $Q \in \mathscr{S}_{s,t}$ and suppose that $a_1 \rightarrow b_k$. If $b_k<c_2$, then $b_k$ bullies either $c_1$ or $c_2$, and we get $\omega_1=1$ or $\omega_2=1$. Hence, condition (\ref{item:1}) holds.
If there is any wrapping from the second row to the third row, then we have $\omega_1=1$ or $2$, which is a contradiction. 

Furthermore, $\omega_2=2$ implies that $b_1 \rightarrow c_2$ and  $b_1$ is not bullied by any $a_i$. Therefore, $c_1<b_1<c_2$ and $b_1<a_1$, and hence $c_1=1$ and $b_1=2$. thus, condition (\ref{item:2}) holds.

If there is any wrapping from the first row to the second row, then we have $a_i \rightarrow b_1$ for some $i$, which is a contradiction, thus proving the remaining half of the condition (\ref{item:3}). It is straightforward to verify that the three conditions imply $\omega_1=3$, $\omega_2=2$, and $c_n=N$.
\end{proof}

From \eqref{Tst>}, recall that $T_1^>(s,t)$ = $\mathbb{P}\{\omega_1=3, \omega_2=2, Q_{3,n}=N\}$. We first compute $T^>(s,t)=T_1^>(s,t)$ and then substitute it into \eqref{PIE4} to solve for $c_{i,j}(n)$ for the case $i>j$. 
Let $\tau_{s,t}$ denote the cardinality of the set $\mathscr{S}_{s,t}$. Then, 
\begin{equation}\label{tauT}
\tau_{s,t}=\binom{n+2s+t}{n,s+t,s}T^>(s,t).
\end{equation}
 \begin{proof}
 [Proof of Theorem~\ref{thm:i>j}]
Let $Q \in \mathscr{S}_{s,t}$. We have $c_1=1$, $b_1=2$, and $c_n=N$ and we have to find the number of ways the remaining elements of $Q$ can be assigned values from the set $\{3, \ldots, N-1\}$ complying with the three conditions of Lemma \ref{lem:cond}.

 Any $b_\ell$ that is bullied by some $a_m$ is larger than $c_2$. Thus, there are at least $s$ integers in the second row that are larger than $c_2$. Therefore, $c_2 \in \{3,4,\ldots, s+t+2\}$ as there can be at most $s$ integers in the first row and at most $t$ integers in the second row that are less than $c_2$. Let $c_2=i$ and let $a_1 \rightarrow b_k$. Then, $c_2 < b_k$ and $b_{k-1}<a_1<b_k$. So, all the numbers from $3$ to $i-1$ are assigned, in order, to $b_2 < \cdots < b_{k-1} < a_1 < \cdots < a_{i-k-1}$. Therefore, $0 \leq i-k-1 \leq s$, that is, $i-s-1 \leq k \leq i-1$. Also, since $k-1$ is the number of integers in the second row that are smaller than $c_2$, it is bounded by $1$ and $t$. Hence, $2 \leq k \leq t+1$.

The ways in which the remaining elements can be assigned values are in bijection with standard Young tableaux of shape $\lambda_{i,k}=(n-2,s+t-k+1,s-i+k+1)$ by Lemma~\ref{lem:syt}, since there is no wrapping between any of the rows. We also have $|\lambda|=N-i$.
This gives us
\begin{equation}\label{tau}
    \tau_{s,t}= \sum_{i=3}^{s+t+2} \sum_{k=\max\{2,i-s-1\}}^{\min\{i-1,t+1\}} f_{\lambda_{i,k}},
\end{equation} 
where $f_\lambda$ is the number of standard Young tableaux of shape $\lambda$.
By the hook length formula for a $3$-row partition~\eqref{3tab}, we have
\begin{multline}
f_{\lambda_{i,k}}  =  \frac{(n-s+i-k-1)(n-s-t+k-2)(t+i-2k+1)}{n(n-1)(s+t-k+2)} \\
 \times\binom{n+2s+t-i}{n-2,s+t-k+1,s+k-i+1} 
 \end{multline}
 Using $N$ for $n+2s+t$ and expanding the multinomial coefficient, we have
\begin{align}\label{hooklen}
f_{\lambda_{i,k}} & =  \frac{(n-s+i-k-1)(n-s-t+k-2)(t+i-2k+1)(N-i)!}{n!(s+t-k+2)!(s-i+k+1)!} \notag\\
&=\frac{(n-s+i-k-1)(n-s-t+k-2)(t+i-2k+1)(N-i+3)!}{n!(s+t-k+2)!(s-i+k+1)!(N-i+3)(N-i+2)(N-i+1)} \notag 
\end{align}

\begin{multline}
f_{\lambda_{i,k}}  =  \frac{(n-s+i-k-1)(n-s-t+k-2)(t+i-2k+1)}{(N-i+3)(N-i+2)(N-i+1)} \\
\times\binom{N-i+3}{n,s+t-k+2,s+k-i+1} 
 \end{multline}

Substituting this formula into ~\eqref{tau} and changing $k$ to $k'=k-2$ and $i$ to $j=i-3$, while plugging in $N=n+2s+t$, we get
\begin{multline}\label{sumi>j}
    \tau_{s,t}=\sum_{j=0}^{s+t-1} \sum_{k'=\text{max }\{j-s,0\}}^{\text{ min }\{j,t-1\}} \frac{(n-s+j-k')(n-s-t+k')(t+j-2k')}{(n+2s+t-j)(n+2s+t-j-1)(n+2s+t-j-2)}\\\times \binom{n+2s+t-j}{n,s+t-k',s-j+k'}.
\end{multline}
Note that when $j\geq s$ and $k'\leq j-s$, the multinomial coefficient in \eqref{sumi>j} becomes 0. Therefore, the index $k'$ can equivalently be summed over the range $0$ to $\min\{j,t-1\}$. Substituting $v$ for $s+t-j-1$ and $u$ for $s-j+k'$ into \eqref{sumi>j}, we get 
\begin{align}\label{tauzeta}
    \tau_{s,t}&=\sum_{v=0}^{s+t-1} \sum_{u=v-t+1}^{\min\{s,v\}} \frac{(n-u)(n+u-s-v-1)(s+v+1-2u)}{(n+s+v+1)(n+s+v)(n+s+v-1)} \binom{n+s+v+1}{n,s+v+1-u,u} \notag\\
    &=\sum_{v=0}^{s+t-1} \sum_{u=v-t+1}^{\min\{s,v\}}\frac{(n+s+v-2)!}{n!(s+v+1)!} \zeta_{s,t} ,
\end{align}
where $\zeta_{s,t}=(n-u)(n+u-s-v-1)(s+v+1-2u)\binom{s+v+1}{u}$. We expand the expression of $\zeta_{s,t}$ to write it as a sum of two parts as follows:
\begin{align}
\zeta_{s,t} =&((n^2-n(s+v+1)+u(s+v+1-2u))\\
            &  \times \bigg((s+v+1-u)\binom{s+v+1}{u}-u\binom{v+s+1}{u}\bigg)\notag \\
      =& (n^2-n(s+v+1))(s+v+1)\bigg(\binom{s+v}{u}-\binom{s+v}{u-1}\bigg) \notag \\
      & +(s+v+1)(s+v)(s+v-1)\bigg(\binom{s+v-2}{u-1}-\binom{s+v-2}{u-2}\bigg). \notag
\end{align}
Plugging $\zeta_{s,t}$ into \eqref{tauzeta}, we have a telescoping sum which computes easily to give
\begin{align*}
    \tau_{s,t}=&\sum_{v=0}^{s+t-1} \frac{(n+s+v-2)!}{n!(s+v+1)!} (n^2-n(s+v+1))(s+v+1)\bigg(\binom{s+v}{s}-\binom{s+v}{v-t}\bigg)\\
    &+\sum_{v=0}^{s+t-1}\frac{(n+s+v-2)!}{n!(s+v+1)!}(s+v+1)(s+v)(s+v-1)\bigg(\binom{s+v-2}{s-1}-\binom{s+v-2}{v-t-1}\bigg)\\
    =&\sum_{v=0}^{s+t-1} \binom{n+s+v}{s+v}\frac{n^2-n(s+v+1)}{(n+s+v)(n+s+v-1)}\bigg(\binom{s+v}{s}-\binom{s+v}{v-t}\bigg)\\
    &+\sum_{v=0}^{s+t-1}\binom{n+s+v-2}{s+v-2}\bigg(\binom{s+v-2}{s-1}-\binom{s+v-2}{v-t-1}\bigg).
\end{align*}
This simplifies to 
\begin{equation}\label{finaltau}
    \tau_{s,t}=\frac{\binom{n+2s+t}{n,s+t,s}}{n+2s+t}\times \frac{nt(s+t)}{(n+s)(n+s+t)} \bigg( -1 + \frac{s}{n} + \frac{(n+s)(n^2+nt-t-s(s+t)-1)}{(n+s-1)(s+t)(n+s+t-1)} \bigg).
\end{equation}
By substituting $\tau_{s,t}=\binom{n+2s+t}{n,s+t,s}T^>(s,t)$ from \eqref{tauT}, we get
\begin{multline}\label{fin}
(n+2s+t)T^>(s,t)= \frac{-nt(s+t)}{(n+s)(n+s+t)}  + \frac{st(s+t)}{(n+s)(n+s+t)}\\ + \frac{nt(n^2+nt-t-s(s+t)-1)}{(n+s-1)(n+s+t)(n+s+t-1)}.
\end{multline}
 Let $i<n$. Denote the terms on the right-hand side of~\eqref{fin} by $A(s,t), B(s,t)$, and $C(s,t)$ respectively. We first compute $\mathscr{A} := A(j-1,i-j)-A(j,i-j-1)-A(j-1,i-j+1)+A(j,i-j)$.
\begin{align*}
        \mathscr{A} & =  \frac{n(i-1)}{n+i-1} \bigg( \frac{j-i}{n+j-1}+\frac{i-j-1}{n+j}\bigg) - \frac{ni}{n+i} \bigg( \frac{j-i-1}{n+j-1}+\frac{i-j}{n+j}\bigg)\\
        & =  \frac{n}{2\binom{n+j}{2}}.
\end{align*}
If we define $\mathscr{B}$ and $\mathscr{C}$ similarly as $\mathscr{A}$ using the inclusion-exclusion formula, we get
$\mathscr{B}  = 
 \frac{n}{2\binom{n+j}{2}}-\frac{n}{2\binom{n+i}{2}}$ and
 $ \mathscr{C}  = 
    \frac{-n}{2\binom{n+i}{2}}$.
    Since $c_{i,j}(n)= \mathscr{A}+\mathscr{B} +\mathscr{C}$, by \eqref{PIE4} we have \[c_{i,j}(n)=\frac{n}{\binom{n+j}{2}}-\frac{n}{\binom{n+i}{2}}.\]
   Note that when $s+t=n$, $T^>(s,t)=0$ by definition. When $j<i=n$,~\eqref{PIE4} becomes
    \[c_{i,j}(n)  = (n+i+j-2)\,T^>(j-1,i-j)-(n+i+j-1)\,T^>(j,i-j-1).\]
    Let $\mathscr{A}'=A(j-1,n-j)-A(j,n-j-1)$ and define $\mathscr{B}'$ and $\mathscr{C}'$ similarly. Simplifying the equations resulting from substituting the expressions for $A(s,t),B(s,t)$ and $C(s,t)$, we get
\begin{align*}
 c_{n,j}(n)&=\mathscr{A}'+\mathscr{B}'+\mathscr{C}'\\
 &= \frac{-1}{2n-1}-\frac{-2n(n-1)}{(n+j-1)(n+j)}+\frac{2n(n-1)}{(n+j-1)(n+j-2)}\\
 &=\frac{n(j+1)}{\binom{n+j}{2}}-\frac{n(j-1)}{\binom{n+j-1}{2}}-\frac{n}{\binom{2n}{2}},
\end{align*} 
  thereby proving Theorem~\ref{thm:i>j}.
\end{proof}

\section[Correlations for the case i < j: Preliminary analysis]{Correlations for the case $i<j$: Preliminary analysis}
\label{sec:i<j}

Let $m_{s,t}=(s,t,n-s-t)$ and $N=n+2s+t$. Let $\mathscr{P}_{s,t}$ be the set of all continuous multiline queues $Q$ of type $\textbf{m}_{s,t}$ that satisfy $\omega_1=2$, $\omega_2=3$, and $Q_{3,n}=N$ where $\omega= \mathcal{B}(Q)$. Let $\theta_{s,t}$ be the cardinality of set $\mathscr{P}_{s,t}$.   
Recall that to find the correlation function $c_{i,j}(n)$ for $i<j$, we need to calculate the probability function $T^<(s,t)=\mathbb{P}\{\omega_1=2,\omega_2=3,Q_{3,n}=N \}$ and then substitute it into \eqref{PIE3}. First, consider the following continuous multiline queue:
\[\begin{array}{cccccccc}
     a_1 & a_2 & \cdots & a_s\\
     b_1 & b_2 & \cdots & \cdots & b_{s+t}\\
     c_1 & c_2 & \cdots & \cdots & \cdots & c_{n-1} & N.\\
     \hline
     x & y & \cdots & \cdots & \cdots & \cdots & \cdots
\end{array}\]
Define $n_{x,y}(s,t,n)$ as the number of continuous multiline queues of type $\textbf{m}_{s,t}$ with $N$ in the last row that project to a word with $x$ and $y$ in the first and the second place respectively as above, where $x,y \in \{1,2,3\}$. Note that $\theta_{s,t}=n_{2,3}(s,t,n)$. Also, let $n_z(s,t,n)$ be the number of continuous multiline queues of type $\textbf{m}_{s,t}$ with $N$ in the last row that project to a word with $z$ in the first place. By similar arguments as in Remark~\ref{rem:rotsym}, $n_z$ also gives the number of continuous multiline queues that project to a word with $z$ in the second place, due to the rotational symmetry of multispecies TASEP. Therefore, we have 
\begin{equation}\label{nx3}
    n_{1,3} + n_{2,3} + n_{3,3} = n_{3},
\end{equation}
for fixed $s,t$ and $n$. Using Lemma~\ref{lem:shift} again, we have
\begin{equation}\label{n3}
n_{3}(s,t,n) = \frac{n-s-t}{N}\binom{n+2s+t}{n,s,s+t},
\end{equation}
since there are $n-s-t$ particles that are labelled $3$. We compute $n_{3,3}(s,t,n)$ in the following lemma.
\begin{lemma}\label{lem:n33}
 $n_{3,3}(s,t) = \binom{N-1}{s} f_{(n-2,s+t)}$. 
 \end{lemma}
 \begin{proof}
 Let $Q$ be a  continuous multiline queue of type $\textbf{m}_{s,t} $ that is counted by $n_{3,3}(s,t,n)$ such that
 \[Q =
\begin{array}{ccccccc}
     a_1 & a_2 & \cdots & a_s\\
     b_1 & b_2 & \cdots & \cdots & b_{s+t}\\
     c_1 & c_2 & \cdots & \cdots & \cdots &c_{n-1} & N\\
     \hline
     3 & 3 & \cdots & \cdots & \cdots & \cdots & \cdots.
\end{array}\]
Here, neither $c_1$ nor $c_2$ is bullied by any element in the second row. This implies that $b_1 > c_2$, and there is no wrapping from the second to the third row. It follows that there are no restrictions on the $a_i$ and hence their values can be chosen from the set $\{1,2, \ldots, N-1 \}$ in $\binom{N-1}{s}$ ways. $c_1$ and $c_2$ take the smallest two integers available after fixing $a_i$'s. Since there are no wrappings from the second to the third row, the configurations formed by the  remaining variables are in 
bijection with standard Young tableaux of shape $(n-2,s+t)$ by Lemma~\ref{lem:syt}. Therefore, they are $f_{(n-2,s+t)}$ in number.
\end{proof}
\begin{remark}\label{rem:n3}

The argument in the proof of Lemma~\ref{lem:n33} can be adapted to give an alternative expression for $n_3(s,t,n)$, namely $n_3(s,t,n)=\binom{N-1}{s} f_{(n-1,s+t)}$, which is equivalent to the previously derived formula in~\eqref{n3}.
\end{remark}
 
Thus, given \eqref{nx3}, it suffices to find $n_{1,3}$ for our analysis, as this allows us to isolate $n_{2,3}(s,t,n)= \theta_{s,t}$ by subtraction. The values of $n_{1,3}(s,t,n)$ for different $s$ and $t$ when $n=5,6$ are shown in the following tables:
\begin{table}[ht]
   \begin{tabular}{| c||c | c | c |}
    \hline
     $s$\textbackslash $t$  & 1 & 2 & 3\\
    \hline \hline
     1   & 9 & 14 & 14\\
    \hline
     2   & 126 & 140 & 0\\
     \hline
     3   & 770 & 0 & 0\\
     \hline
    \end{tabular}
    \quad
    \begin{tabular}{| c||c | c | c | c |}
    \hline
     $s$\textbackslash $t$ & 1 & 2 & 3 & 4\\
    \hline \hline
     1  & 14 & 28 & 42 & 42\\
    \hline
     2  & 280 & 462 & 504 & 0\\
     \hline
     3  & 2772 & 3276 & 0 & 0\\
     \hline
     4  & 15288  & 0 & 0 & 0\\
     \hline
    \end{tabular}
    \vspace{0.2in}
    \caption{Data for $n_{1,3}(s,t,n)$ for $n=5,6$.}
    \label{tab:n13}
\end{table}

By studying the values for $n_{1,3}$ for different $s,t$ and $n$, we formulate the following expression for $n_{1,3}$, which we prove in Section~\ref{sec:n13}.
\begin{proposition}\label{prop:n13}
For $s , t \geq 1$ and $n > s+t$, we have 
\begin{equation}
    n_{1,3}(s,t,n)=\binom{N-1}{s-1} f_{(n-1,s+t)}.
\end{equation}
\end{proposition}

\begin{proof}[Proof of Theorem~\ref{thm:i<j}]

From straightforward calculations 
 using equations ~\eqref{nx3}, ~\eqref{n3}, Lemma~\ref{lem:n33} and Proposition~\ref{prop:n13}, we have
\[n_{2,3}(s,t,n)=\frac{1}{N}\binom{n+2s+t}{n,s,s+t}\Bigg(\frac{n+t}{\binom{n+s+t}{s+t}}f_{(n-1,s+t)}-\frac{n+s+t}{\binom{n+s+t}{s+t}} f_{(n-2,s+t)}\Bigg).\]
Since $n_{2,3}(s,t,n)= \theta_{s,t}= \binom{n+2s+t}{n,s,s+t}T^<(s,t)$, we have 
\begin{equation}\label{Tst}
   (n+2s+t)T^<(s,t)= \frac{n+t}{\binom{n+s+t}{s+t}}f_{(n-1,s+t)}-\frac{n+s+t}{\binom{n+s+t}{s+t}} f_{(n-2,s+t)}.
\end{equation}
The proof is completed by substituting \eqref{Tst} into \eqref{PIE3}. First let $i+1 < j$. Denote the terms of the right-hand side of~\eqref{Tst} by $C(s,t)$ and $D(s,t)$ respectively.
We first compute $\mathscr{C}:=C(i-1,j-i)-C(i,j-i-1)-C(i-1,j-i+1)+C(i,j-i)$.
\begin{align*}
        \mathscr{C} =&  \frac{f_{(n-1,j-1)}}{\binom{n+j-1}{j-1}} (n+j-i-n-j+i+1)
        -\frac{f_{(n-1,j)}}{\binom{n+j}{j}} (n+j-i-n-j+i-1)\\
        =&  \frac{n}{\binom{n+j}{2}}.
\end{align*}
If we similarly define $\mathscr{D}:=D(i-1,j-i)-D(i,j-i-1)-D(i-1,j-i+1)+D(i,j-i)$, we get $\mathscr{D}=0$. Since $c_{i,j}(n)= \mathscr{C}+\mathscr{D}$ by \eqref{PIE3}, we have $c_{i,j}(n)=\frac{n}{\binom{n+j}{2}}$.

Note that $T^<(s,0)=0$ by definition. Therefore when $i+1=j$,~\eqref{PIE3} becomes
\begin{multline*}
    c_{j-1,j}(n)  = (n+2j-3)\,T^<(j-2,1)-(n+2j-2)T^<(j-2,2)\\ + (n+2j-1)\,T^<(j-1,1).
\end{multline*}    
    
    Let $\mathscr{C}'=C(j-2,1)-C(j-2,2)+C(j-1,1)$ and define $\mathscr{D}'$ similarly. Then,
\begin{align*}
\mathscr{C}'&= \frac{n+1}{\binom{n+j-1}{j-1}}f_{(n-1,j-1)} - \frac{n+2}{\binom{n+j}{j}}f_{(n-1,j)} + \frac{n+1}{\binom{n+j}{j}}f_{(n-1,j)},\\
\mathscr{D}'&=-\frac{n+j-1}{\binom{n+j-1}{j-1}}f_{(n-2,j-1)}
\end{align*} 
Thus, $c_{j-1,j}(n) = \mathscr{C}'+\mathscr{D}'
=  \frac{ni}{\binom{n+i}{2}}+ \frac{n}{\binom{n+j}{2}},$ completing the proof.
\end{proof}

\section{Proof of Theorem~\ref{thm:i<j} and Resolution of the Aas--Linusson Conjecture}\label{sec:progress}
\subsection{Direct Proof of Proposition~\ref{prop:n13}.}\label{sec:n13}
To compute $n_{1,3}(s,t,n)$, we count the number of continuous multiline queues with the following configuration. Recall that $N=n+2s+t$ while $a_i,b_i$, and $c_i$ are distinct integers from the set $[N]$.
 \[\begin{array}{ccccccccc}
    & a_1 & a_2 & \cdots & a_s\\
    & b_1 & b_2 & \cdots & \cdots & b_{s+t}\\
    & c_1 & c_2 & \cdots & \cdots & \cdots &c_{n-1} &N.\\
     \hline
    \omega= &1 & 3 & \cdots & \cdots & \cdots &\cdots & \cdots
\end{array}\]

Since $\omega_2=3$, $c_2$ cannot be bullied by any $b_i$. Therefore, there can be at most one wrapping from the second row to the third row. These configurations can be classified into the following two distinct types:
\begin{enumerate}
     \item There is no wrapping from the second row.
     \item Only $b_{s+t}$ wraps around and bullies $c_1$.
 \end{enumerate}
 Let us denote the number of the continuous multiline queues from the two cases by $\alpha_{1,3}(s,t,n)$ and $\beta_{1,3}(s,t,n)$ respectively. We will enumerate them separately. Then, we have
 \begin{align}\label{n13sum}
     n_{1,3}(s,t,n) =\alpha_{1,3}(s,t,n) + \beta_{1,3}(s,t,n).
 \end{align} 

\begin{proposition}\label{prop:nowrapping}
 A continuous multiline queue of type $\textbf{m}_{s,t}$, where there is no wrapping from the second row,  projects to a word $\omega$ with $\omega_1=1$ and $\omega_2=3$ if and only if
 \begin{enumerate}
     \item\label{item:it1} there exists $1\leq i \leq s$ such that $a_i \rightarrow b_1 \rightarrow c_1$, and
     \item\label{item:it2} $b_2 > c_2$.
 \end{enumerate}
 \end{proposition}
\begin{proof}
Let $Q$ be a continuous multiline queue with no wrapping from the second row, satisfying (\ref{item:it1}) and (\ref{item:it2}), and let $\omega$ be the projected word of $Q$. The reverse implication is straightforward. We proceed to prove the forward implication. 
Note that $b_2<c_2$ implies that $c_2$ is bullied by either $b_1$ or $b_2$, resulting in $\omega_2<3$, which is a contradiction. Therefore, $b_2>c_2$.

Since there is no wrapping from the second row to the third row, $\omega_1=1$ can only occur when there exist $a_i$ and $b_j$ such that $a_i \rightarrow b_j \rightarrow c_1$ for some $i$, meaning $b_j < c_1$. Then, $b_j < c_1 < c_2 < b_2$ implies that $j=1$.
\end{proof}

 \begin{theorem}\label{thm:part1}
  Let $\alpha_{1,3}(s,t,n)$ denote the number of continuous multiline queues of type $\textbf{m}_{s,t}$, with $N$ in the last row, where there is no wrapping from the second row, and the projected word $\omega$ satisfies $\omega_1=1$ and $\omega_2=3$. Then, \[\alpha_{1,3}(s,t)= \bigg(\binom{N-2}{s-1}-\binom{N-2}{s-3}\bigg)f_{(n-1,s+t)}.\]
  \end{theorem}
 \begin{proof}
 Let the continuous multiline queues counted by $\alpha_{1,3}(s,t,n)$ exhibit the following configuration:
 \[\begin{array}{ccccccccc}
    & a_1 & a_2 & \cdots & a_s\\
    & b_1 & b_2 & \cdots & \cdots & b_{s+t}\\
    & c_1 & c_2 & \cdots & \cdots & \cdots &c_{n-1} &N\\
     \hline
    \omega= &1 & 3 & \cdots & \cdots & \cdots &\cdots & \cdots
\end{array}\]
    
Let us first assume that $a_1 < b_1$. Coupled with $b_1 < c_1$ from the proof of Proposition~\ref{prop:nowrapping}, this implies $a_1=1$ and $a_1 \rightarrow b_1 \rightarrow c_1$, which gives $\omega_1=1$.
 The remaining $a_i$'s take increasing values between $2$ and $N-1$ in $\binom{N-2}{s-1}$ ways. Additionally, we have  $c_1 < c_2 < b_2$. Thus, $b_1, c_1$, and $c_2$ are assigned the three smallest values after eliminating those selected by the $a_i$'s. Since there is no wrapping from either of the rows, these configurations are in bijection with standard Young tableaux of shape $\lambda=(n-2,s+t-1)$ (see Lemma~\ref{lem:syt}). This gives us $\binom{N-2}{s-1}f_{(n-2,s+t-1)}$ continuous multiline queues that contribute to $\alpha_{1,3}(s,t,n)$ for the case $a_1 < b_1$. 

Now, let us assume that $a_1 > b_1$. Then, there exists an $a_i$ such that $a_i \xrightarrow{W} b_1$ by Proposition~\ref{prop:nowrapping}(\ref{item:it1}). The inequalities $b_1 < c_1$ and $b_1<a_1$ imply that $b_1=1$. We first focus on finding the number of continuous multiline queues where the only constraints are $b_1=1, c_n=N$, and $c_2<b_2$, with no wrapping from the second row to the third row. For these, 
  $b_1 \rightarrow c_1$, and we get $\omega_1 \in \{1,2\}$ and $\omega_2=3$. Observe that the $a_i$'s can take any value other than $1$ and $N$. Then, $c_1$ and $c_2$ are assigned the smallest two values after eliminating $1$ and the integers selected by the $a_i$'s.  Also, the configurations formed by the remaining $b_i$'s and $c_i$'s satisfy the following inequalities:
\[\begin{array}{ccccccccccccc}
      & &b_{2} & < & \ldots & < &  b_{s+t-j} & < &\ldots & < & b_{s+t}\\
      &  & \wedge &  & \cdots & & \wedge &  & \cdots & & \wedge\\
      c_3 & < &  \ldots & < & \ldots & < & c_{n-j} & < & \ldots & < & c_n,
\end{array}\]
and hence they can be arranged in $f_{(n-2,s+t-1)}$ ways.  The required number is given by $\binom{N-2}{s}f_{(n-2,s+t-1)}$. To eliminate the continuous multiline queues where  $\omega_1=2$ and $\omega_2=3$, we need to subtract the number of continuous multiline queues where $b_1=1$, $b_2<c_2$ with no wrapping in any row, from the number $\binom{N-2}{s}f_{(n-2,s+t-1)}$. To do that, let $c_2 = k+2$ for some $k \leq s$. Then, there are $k$ values of $a_i$ that are smaller than $c_2$, and there are $k+1$ ways to assign values to $c_1,a_1,\cdots a_k$. The remaining elements of the continuous multiline queue satisfy the following inequalities:
\[\begin{array}{ccccccccccccc}
     & & & & & & a_{k+1} & < & a_{k+2} & < & \cdots & < & a_s\\
     & & & & & & \wedge & & \wedge & & \cdots & & \wedge\\
     & &b_{2} & < & \cdots & < &  b_{t-k-1} & < & b_{t-k} & < &\cdots & < & b_{s+t}\\
     & & \wedge & & \ldots &  & \wedge & & \wedge & & \cdots & & \wedge\\
     c_3 & < & \cdots& & \cdots & & \cdots &  &\cdots & & \cdots & < & c_n.
\end{array}\]
Such configurations are in bijection with standard Young tableaux of shape $\lambda_k=(n-2,s+t-1,s-k)$, which are $f_{(n-2,s+t-1,s-k)}$ in number. Thus, the number of continuous multiline queues contributing to $\alpha_{1,3}(s,t,n)$ where $a_1>b_1$ is given by 
\begin{equation*}
\binom{N-2}{s}f_{(n-2,s+t-1)}-\sum_{k=0}^s (k+1)f_{(n-2,s+t-1,s-k)}.
 \end{equation*}
 Adding this to $\binom{N-2}{s-1}f_{(n-2,s+t-1)}$ for the case $a_1<b_1$, we get
\begin{multline}\label{alpha13}
    \alpha_{1,3}(s,t,n) = \binom{N-2}{s-1}f_{(n-2,s+t-1)}+\binom{N-2}{s}f_{(n-2,s+t-1)}\\
    -\frac{f_{(n-2,s+t-1)}}{n(s+t)}\sum_{k=0}^s (k+1)(t+k)(n-s+k)\binom{N-k-3}{s-k}.
\end{multline}
 Summing $(k+1)(t+k)(n-s+k)\binom{N-k-3}{s-k}$ over $k=0$ to $s$, we get
\begin{multline}\label{longsum}
    \frac{(N-1)!}{s!(n+s+t+1)!}(s((n+2)^2t+2n(n+2)+t^2-t^3-4)-s^2 ((n+2t+1)t-6)\\
    -s^3(t+2)+nt(n+t)(n+t+1)). 
\end{multline}
Let the large expression inside the parentheses in \eqref{longsum} be denoted by $(\ast)$. Simplifying the right-hand side of \eqref{alpha13}, we have

\begin{align*}
\alpha_{1,3}(s,t,n) &= \bigg(1-\frac{(*)}{n(s+t)(n+s+t)(n+s+t+1)}\bigg)\binom{N-1}{s}f_{(n-2,s+t-1)}\\
    &=\frac{s(n+t+2)(n+s+t-1)(n+s+t-2)}{n(s+t)(n+s+t)(n+s+t+1)}\binom{N-1}{s}f_{(n-2,s+t-1)}\\
    &=\frac{s(n+t+2)}{(n+s+t)(n+s+t+1)}\binom{N-1}{s}f_{(n-1,s+t)}\\
    &=\bigg(\binom{N-2}{s-1}-\binom{N-2}{s-3}\bigg)f_{(n-1,s+t)}.
\end{align*}
\end{proof}
 
\begin{proposition}\label{prop:wrapping}
 A continuous multiline queue of type $\textbf{m}_{s,t}$ with $N$ in the last row, such that there is exactly one wrapping from the second row, projects to a word $\omega$ with $\omega_1=1$ and $\omega_2=3$ if and only if
\begin{enumerate}
     \item there exists $ i < s$ such that $a_i \rightarrow b_{s+t-1} \rightarrow c_n$ and $a_{i+1} \rightarrow b_{s+t} \xrightarrow{W} c_1$, and
     \item $b_1 > c_2$.
\end{enumerate}
 \end{proposition}
\begin{proof}
Let $Q$ be a continuous multiline queue of type $\textbf{m}_{s,t}$ with $N$ in the last row and exactly one wrapping from the second row. Let $\omega=\mathcal{B}(Q)$ be the projected word such that $\omega_1=1$ and $\omega_2=3$. If $b_1<c_2$ in $Q$, then $c_2$ is bullied by at least one $b_i$ (either by $b_1$ or by the only wrapping), which results in $\omega_2<3$, a contradiction. Therefore, $b_1 >c_2$.

Then, $\omega_1=1$ implies that $c_1$ is bullied via wrapping. Let $j$ and $v$ be integers such that $a_j \rightarrow b_v \xrightarrow{W} c_1$. If $v<s+t$, then there exists $w>v$ such that $b_w \xrightarrow{W} c_2$, once again leading to a contradiction. Therefore, we conclude that $v=s+t$.

Next, since $a_j \rightarrow b_{s+t} \xrightarrow{W} c_1$,  there exist integers $i<j$ and $u <s+t$ such that $a_i\rightarrow b_u \rightarrow c_n$, giving $\omega_n=1$. Otherwise, $b_{s+t}$ bullies $c_n=N$, and there is no wrapping from the second row to the third. Since $b_u \rightarrow c_n$, the entries $b_{u+1}, b_{u+2}, \cdots, b_{s+t}$ all wrap around to the third row. But since there can be only one wrapping, we must have $u+1=s+t$. Additionally, because $a_i \rightarrow b_{s+t-1}$, $a_k$ wraps around to the second row for all $k>i+1$, thereby proving that $j={i+1}$. The reverse implications can be easily verified.
\end{proof}
 Recall that the number of continuous multiline queues satisfying Proposition~\ref{prop:wrapping} is represented by $\beta_{1,3}$. The values of $\beta_{1,3}(s,t,n)$ for different $s$ and $t$ for $n=5,6$ are shown in  Table~\ref{tab:beta13}. 
\begin{table}[ht]
    \centering
    \begin{tabular}{| c||c | c | c |}
    \hline
     $s$\textbackslash $t$& 0 & 1 & 2\\
    \hline \hline
     2   & 9 & 14 & 14 \\
    \hline
     3   & 140 & 154 & 0\\
     \hline
     4   & 924  & 0 & 0\\
     \hline
    \end{tabular}
    \quad
    \begin{tabular}{| c||c | c | c | c |}
    \hline
     $s$\textbackslash $t$& 0 & 1 & 2 & 3\\
    \hline \hline
     2   & 14 & 28 & 42 & 42 \\
    \hline
     3   & 280 & 504 & 546 & 0\\
     \hline
     4   & 3276 & 3822  & 0 & 0\\
     \hline
     5   & 19110 & 0 & 0 & 0\\
     \hline
    \end{tabular}
    \vspace{0.2in}
    \caption{Data for $\beta_{1,3}(s,t,n)$ for n=5,6}
    \label{tab:beta13}
\end{table}

Observing the data in Table~\ref{tab:beta13}, we formulate the following expression for $\beta_{13}$:
\begin{proposition}\label{prop:beta}
$\beta_{1,3}(s,t,n)=
 \binom{N-1}{s-2}f_{(n-1,s+t)}$.
\end{proposition}
\begin{remark}

 It is interesting to note that, despite its seemingly simple formulation, the techniques we used to prove the earlier cases do not work here, as there is no straightforward bijection available to prove Proposition~\ref{prop:beta}. For now, we independently prove Proposition~\ref{prop:n13} using the properties of continuous multiline queues counted by $\beta_{1,3}$. We will return to computing $\beta_{1,3}$ using alternative methods at the end of this section.
\end{remark}

We first prove that $\beta_{1,3}(s,t,n)$ satisfies a simple recurrence.
\begin{lemma}\label{lem:rec} For $s,t \geq 2$ and $s+t < n$, we have:
\begin{equation}\label{recbeta}
\beta_{1,3}(s,t,n) = \beta_{1,3}(s-1,t+1,n) + \beta_{1,3}(s,t-1,n) + \beta_{1,3}(s,t,n-1).
\end{equation}
\end{lemma}
\begin{proof}
Consider a continuous multiline queue $Q$ that satisfies the conditions from Proposition~\ref{prop:wrapping}. Let $c_2 = k+2$ for some $k \leq s$. Then, $Q$ has the following configuration: 
\[\begin{array}{ccccccccc}
     a_1 & a_2 & \cdots & a_k &\cdots& a_s\\
     b_1 & b_2 & \cdots & b_k &\cdots & \cdots & b_{s+t}\\
     c_1 & k+2 & \cdots & \cdots & \cdots & \cdots &\cdots &c_{n-1} &N\\
     \hline
     1 & 3 & \cdots & \cdots  &\cdots & \cdots &\cdots& \cdots &1
\end{array}\]
Since $b_1 > c_2$, there are only $\binom{k+1}{k} = k+1$ ways to assign values to $c_1,a_1,\ldots,a_k$ from the set $[k+1]$. Thus, for each $u \leq k, \, a_u \rightarrow b_u$.
Given that the rows are strictly increasing, exactly one of the following cases is true:
\begin{enumerate}
    \item $a_{k+1}=k+3$\\
    In this case, $a_{k+1}$ bullies $b_{k+1}$, so it does not bully $b_{s+t-1}$ or $b_{s+t}$. Deleting $a_{k+1}$ and subtracting $1$ from all values greater than $k+3$ does not affect the bully paths that include $b_{s+t-1}$ or $b_{s+t}$. In the new projected word, there is one less $1$, and one more $2$. Therefore, there are $\beta_{1,3}(s-1,t+1,n)$ such continuous multiline queues.
    \item $b_1=k+3$\\
    Since, $b_1$ bullies $c_3$, deleting $b_1$ and subtracting $1$ from all values greater than $k+3$ in the continuous multiline queue does not affect the bully paths that include $b_{s+t-1}$ or $b_{s+t}$. In the new projected word, there is one less $2$, and one more $3$. Therefore, there are $\beta_{1,3}(s,t-1,n)$ such continuous multiline queues.
    \item $c_3=k+3$\\
    Finally, in this case, $c_3$ is not bullied by any $b_u$ because $b_1>c_3$, and there is exactly one wrapping from the second to the third row. Deleting $c_3$ and subtracting $1$ from all values greater than $k+3$ does not affect any bully path, and the length of the resulting projected is reduced by one. There are $\beta_{1,3}(s,t,n-1)$ such continuous multiline queues.
\end{enumerate}
Therefore, $\beta_{1,3}(s,t,n)$ is obtained by adding the numbers in each of the above cases. 
\end{proof}

We can now verify the equation 
\[\alpha_{1,3}(s,t,n) = \alpha_{1,3}(s-1,t+1,n) + \alpha_{1,3}(s,t-1,n) + \alpha_{1,3}(s,t,n-1),\] for $s,t \geq 2$ and $n>s+t$ by plugging in the value of $\alpha_{1,3}(s,t,n)$ from Theorem \ref{thm:part1}. This along with Lemma \ref{lem:rec}  gives the recurrence relation \begin{equation*}
    n_{1,3}(s,t,n) = n_{1,3}(s-1,t+1,n) + n_{1,3}(s,t-1,n) + n_{1,3}(s,t,n-1),
\end{equation*}
for $s,t\geq 2$ and $n>s+t$.
Let $P(s,t,n)$ denote the product $\binom{N-1}{s-1} f_{n-1,s+t}$ from Proposition~\ref{prop:n13}. We have
 \[P(s,t,n) = P(s-1,t+1,n) + P(s,t-1,n)+ P(s,t,n-1),\]
using Pascal's rule and the hook length recurrence relation $f_{(a,b)}=f_{(a-1,b)} + f_{(a,b-1)}$ where $a\geq b>1$ (see Remark~\ref{rem:hookrec}). As a result, $P(s,t,n)$ satisfies the same recurrence relation as $n_{1,3}(s,t,n)$. Additionally, $n_{1,3}(s,t,s+t)=0$ because there is no $3$ in the projected word.
By \eqref{n13sum}, we also have 
\[
n_{1,3}(1,t,n)=\alpha_{1,3}(1,t,n)+0
=f_{(n-1,t+1)},
\]
because $\beta_{1,3}(1,t,n)=0$ for all $t$ and $n$. This holds
because, by Proposition~\ref{prop:wrapping}(1), we need $s>1$ to ensure exactly one wrapping from the second row to the third row. Thus, proving the initial condition $n_{1,3}(s,1,n)=P(s,1,n)$ completes the proof of Proposition~\ref{prop:n13}. To this end, we have the following result.
\begin{proposition}\label{prop:incon}
For $s,n$ such that $1 < s+1 <n$, we have
    \begin{equation}\label{n13}
    n_{1,3}(s,1,n) = \binom{n+2s}{s-1}f_{(n-1,s+1)}.
    \end{equation} 
\end{proposition}

\begin{proof}
Recall that $\textbf{m}_{s,t}=(s,t,n-s-t)$, and that $n_{x,y}(s,t,n)$ counts the number of continuous multiline queues $Q$ of type $\textbf{m}_{s,t}$, with $N$ in the last row, such that $Q$ projects to a word $\omega$ where $\omega_1=x$ and $\omega_2=y$. 
\[Q=\begin{array}{ccccccccc}
     a_1 & a_2 & \cdots &  a_s\\
     b_1 & b_2 & \cdots & \cdots & b_{s+t}\\
     c_1 & c_2 &  \cdots & \cdots & \cdots & c_{n-1}& N.\\
     \hline
     x & y & \cdots & \cdots  &\cdots & \omega_{n-1} & \omega_n
\end{array}\]
Let $\mu_{x,y}$ be the probability that a particle labelled $x$ is immediately followed by a particle labelled $y$ in the continuous TASEP of type $\textbf{m}_{s,t}$. 
 Then, by Lemma~\ref{lem:rot}, 
 \[\mu_{x,y}=\frac{n+2s+t}{\binom{n+2s+t}{s,s+t,n}}n_{x,y}.\]
Consider $Q'$, a continuous multiline queue of type $\textbf{m}_u=(u,n-u)$ such that the largest integer $N'=n+u$ is in the last row. Recall from Section~\ref{sec:twospecies}, that $\delta_{c,d}(u,n-u)$ counts the number of continuous multiline queues $Q'$ with $N'$ in the last row, such that $Q'$ projects to a word $\omega'$ where $\omega'_1=c$ and $\omega'_2=d$. 
\[Q'=\begin{array}{ccccccccc}
     a_1 & a_2 & \cdots &  a_u\\
     b_1 & b_2  & \cdots & \cdots & b_{n-1}& N'\\
     \hline
     c & d & \cdots  &\cdots & \omega'_{n-1} & \omega'_n
\end{array}.\]
From \eqref{del22}, we have that $\delta_{2,2}(u,n-u)=f_{(n-2,u)}$. Let $\epsilon_{c,d}(n)$ denote the probability that a particle labelled $c$ is immediately followed by one labelled $d$ in a continuous two-species TASEP of type $\textbf{m}_u$. Then, again by Lemma~\ref{lem:rot},
 \[\epsilon_{c,d}(n)=\frac{n+u}{\binom{n+u}{u,n}}\delta_{c,d}(u,n-u).\]
We can define a lumping from the continuous three-species TASEP of type $\textbf{m}_{s,t}$ to a continuous two-species TASEP as follows. Let $\Omega_{s,t}$ and $\Omega_s$ be the set of labelled words on a ring of type $\textbf{m}_{s,t}$ and $\textbf{m}_s=(s,n-s)$ respectively. Let $f :\Omega_{s,t} \rightarrow \Omega_s$ be a map defined as follows:
\[f(\omega_1, \ldots, \omega_n)=(f(\omega_1),\ldots,f(\omega_n)),\]
where \[f(i)=\begin{cases}
1, & \text{if }i=1,\\
2, & \text{if }i=2,3.
\end{cases}\]
\begin{table}[ht]
    \centering
    \begin{tabular}{| c||c | c | c |}
    \hline
     $x$\textbackslash $y$ & 1 & 2 & 3\\
    \hline \hline
     1 & \cellcolor{yellow!25}$\mu_{1,1}$& \cellcolor{green!25}$\mu_{1,2}$& \cellcolor{green!25}$\mu_{1,3}$  \\
     \hline
     2 & \cellcolor{red!25}$\mu_{2,1}$& \cellcolor{blue!25}$\mu_{2,2}$& \cellcolor{blue!25}$\mu_{2,3}$\\
     \hline
     3 & \cellcolor{red!25}$\mu_{3,1}$& \cellcolor{blue!25}$\mu_{3,2}$& \cellcolor{blue!25}$\mu_{3,3}$\\
     \hline
\end{tabular}
\vspace{0.2in}
    \caption{The table contains the correlations of two adjacent particles in $\Omega_{s,t}$. The table is divided into four parts, and entries of the yellow, green, red, and blue sections contribute to the correlations  $\epsilon_{1,1},\epsilon_{1,2},\epsilon_{2,1}$ and $\epsilon_{2,2}$ respectively in $\Omega_s$. }
    \label{tab:lump}
\end{table}

\noindent
By lumping the Markov process, we have 
\begin{equation*}
 \epsilon_{2,2}=\{\mu_{2,2}+\mu_{3,2}+\mu_{2,3}+\mu_{3,3}\}.   
\end{equation*}
That is,
\begin{equation}\label{system}
\frac{n+s}{\binom{n+s}{s}}\delta_{2,2} = \frac{n+2s+t}{\binom{n+2s+t}{n,s,s+t}}\{n_{2,2}+n_{3,2}+n_{2,3}+n_{3,3}\}.
\end{equation}
Let $t=1$. Then, $n_{2,2}=0$, because there is only one particle with label $2$ in the continuous three-species TASEP of type $\textbf{m}_{s,1}$. Note that $n_{3,2}(s,t,n)=\tau_{s,t}$ from Section~\ref{sec:i>j}. Thus, substituting $t=1$ into \eqref{finaltau}, we obtain  
\begin{equation}\label{n32}
    \frac{n+2s+1}{\binom{n+2s+1}{n,s,s+1}}n_{3,2}(s,1,n)= \frac{(n-s-1)(n-s+1)}{(n+s-1)(n+s+1)}.
\end{equation}
 We also have 
\begin{equation}\label{ndot3}
\{n_{1,3}+n_{2,3}+n_{3,3}\}(s,1,n) = \frac{n-s-1}{n+2s+1}\binom{n+2s+1}{n,s,s+1},
\end{equation}
by substituting $t=1$ into \eqref{nx3}.
Solving~\eqref{ndot3} for $n_{1,3}$ using \eqref{system} and \eqref{n32} gives 
\begin{align*}
n_{1,3}(s,1,n)&=\frac{s(n-s-1)}{(n+s+1)(n+2s+1)}\binom{n+2s+1}{n,s,s+1}
\\
&=\binom{n+2s}{s-1}f_{(n-1,s+1)},
\end{align*}
thereby proving the result.
\end{proof}
Now, we can also prove Proposition~\ref{prop:beta} directly using Proposition~\ref{prop:n13} and Theorem~\ref{thm:part1} as follows.
\begin{proof}[Proof of Proposition~\ref{prop:beta}]
\begin{align*}
   \beta_{1,3}(s,t,n)&=n_{1,3}(s,t,n)-\alpha_{1,3}(s,t,n)\\
   &=\bigg(\binom{N-1}{s-1}-\binom{N-2}{s-1}+\binom{N-2}{s-3}\bigg)f_{(n-1,s+t)}\\
   &=\binom{N-1}{s-2}f_{(n-1,s+t)}.
\end{align*}\end{proof}
\subsection{An Alternative Approach via First Principles}\label{sec:betaproof}

In this section, we additionally describe the developments made towards a direct proof of Proposition~\ref{prop:beta} using the first principles.
Recall the recurrence relation \eqref{recbeta}. The equation holds true for $s,t\geq 2$ and $s+t\leq n$. It is easy to verify that the product $\binom{N-1}{s-2}f_{(n-1,s+t)}$ satisfies the same recurrence relation as $\beta_{1,3}(s,t,n)$. Thus, it is sufficient to show that the initial conditions are the same for both quantities. The conditions are as follows:
\begin{align}\label{eq:beta1}
    \beta_{1,3}(1,t,n)&=0,\notag \\
    \beta_{1,3}(s,t,s+t)&=0,\notag \\
    \beta_{1,3}(s,1,n)&=\binom{N-1}{s-2}f_{(n-1,s+1)}.
\end{align}
The first two initial conditions are straightforward. We now provide a formula for $\beta_{1,3}$ for $t=1$.

\begin{lemma}\label{lem:sum}
\[\beta_{1,3}(s,1,n) = \sum_{\ell=2}^{s}\gamma^\ell(s,n),\] where $\gamma^\ell(s,n)$ is the number of continuous multiline queues of type $\textbf{m}_{s,1}$ with $c_2=\ell, c_n=N$ and $b_1>c_2$, such that the projected words have $1$ and $3$ in the first two places respectively.   
\end{lemma}
\begin{proof}
Consider the following continuous multiline queue $Q$ of type $\textbf{m}_{s,1}$ that is counted by $\beta_{1,3}(s,1,n)$:
\[Q= \begin{array}{cccccccc}
     a_1 & a_2 & \cdots  &\cdots& a_s\\
     b_1 & b_2 & \cdots  &\cdots & b_s & b_{s+1}\\
     c_1 & c_2 & \cdots & \cdots & \cdots &\cdots &c_{n-1} &N.\\
     \hline
     1 & 3 & \cdots & \cdots & \cdots &\cdots& \cdots &1
\end{array}\]
Let $c_2$ be equal to $\ell$, which is greater than $1$. According to  Proposition~\ref{prop:wrapping}, $c_2<b_1$. Therefore, the values $\{1,\ldots,\ell-1\}$ are assigned to $c_1,a_1,\ldots,a_{\ell-2}$. Since there exists $i$ such that $a_i\rightarrow b_s$ and $a_{i+1}\rightarrow b_{s+1}$, $\ell$ ranges from $\{2,\ldots, s\}$. The set of all continuous multiline queues counted by $\beta_{1,3}(s,1,n)$ can be divided into smaller sets depending on the value of $\ell$. We denote the number of such continuous multiline queues that have $c_2=\ell$ as $\gamma^\ell(s,n)$. Summing over all possible values of $\ell$, we obtain the required expression.
\end{proof}

Next, we derive a formula for $\gamma^\ell$ from first principles.  First, consider a skew shape $\lambda/\mu$, where $\lambda$ and $\mu$ are partitions such that $\mu \subseteq \lambda$ are in containment order. Recall from Section~\ref{sec:syt}, that the number of standard Young tableaux of a skew shape $\lambda/\mu$ is given by $f_{\lambda/\mu}$, which can be computed using \eqref{frob}.
\begin{theorem}\label{thm:gammaexp}For $2 \leq \ell \leq s$, we have
\begin{multline}\label{gammaexpr} 
\gamma^\ell (s,n) = (\ell-1)\sum_{i=\ell-1}^{s-1} \,\sum_{j=\ell-2}^{i-2}\, \sum_{k=2}^{n-s+i-1} f_{(n-s+i-3,i-2,j-\ell+2)/(n-s+i-k-1)}  \\ 
\times (f_{(n+i-k-j,s-j,s-j)/(i-j+1,i-j-1)}-f_{(n+i-k-j-1,s-j,s-j)/(i-j,i-j-1)}).
\end{multline}
\end{theorem}
\begin{proof}
Let $c_2=\ell$. Then, $\gamma^\ell(s,n)$ counts the number of continuous multiline queues $Q$ with the following configuration.
\[Q=\begin{array}{ccccccccc}
     a_1 & a_2  &\cdots& a_s\\
     b_1 & b_2 & \cdots  & b_s & b_{s+1}\\
     c_1 & \ell  &\cdots & \cdots  &\cdots & N,\\
     \hline
     1 & 3 & \cdots & \cdots & \cdots & 1
\end{array}\]
where $b_1>c_2$, such that $c_1,a_1,\ldots a_{\ell-2} \in [\ell-1]$. We can choose $c_1$ in $(\ell-1)$ different ways, which determines the values of $a_1$ to $a_{\ell-2}$. These values must satisfy $a_u<b_u$ for $1\leq u \leq \ell-2$, implying that  $a_u \rightarrow b_u$ for each $u$.

Since $Q$ is of type $(s,1,n-s-1)$, there is exactly one $b_i$ that is not bullied by any $a_j$, for some $i \in \{\ell-1,\ldots, s-1\}$.
For a fixed $i$, we can split the set of elements of a continuous multiline queue into two sets based on their relation with $b_i$. Let $j$ and $k$ be the largest integers such that $a_j < b_i$ and $c_k < b_i$ respectively. The value of $j$ lies between $\ell-2$ and $i-1$ by the choice of $i$ and $j$. Additionally, $n-k>s-i$ ensures that there is no more than one wrapping from the second row to the third row, which implies that $k$ lies between $0$ and $n-s+i-1$. For the continuous multiline queues with at most one wrapping from the second row to the third row, the following inequalities hold for fixed $i,j$ and $k$  according to Proposition~\ref{prop:wrapping}.

\[\begin{array}{ccccccccccccc}
     & &a_{\ell-1}  & < & \ldots & <  & a_j\\
     & &\wedge & & \cdots & & \wedge\\
     &b_{1}  < & \ldots & < &\ldots & < & b_{i-1}\\
     & \wedge& \cdots &  & \wedge & & & \\
     c_3 <&\ldots & \ldots & < & c_k
\end{array}
< b_i < 
\begin{array}{ccccccccccccc}
        &a_{j+1}   < & \ldots &\ldots  < & \ldots & <a_s\\
       &\wedge  &  &     \wedge & \wedge\\
       &b_{i+1}  <  &\ldots  & <b_s  < & b_{s+1}\\
        &\wedge  &  & \wedge & & \\
     c_{k+1}  <  &\cdots  <& \cdots  & c_n.
\end{array}\]

These arrangements are counted by the product of hook length formulas of appropriate skew shapes, that is, by
\begin{equation}\label{prod1}
    f_{\lambda_1/\mu_1}.f_{\lambda_2/\mu_2},
\end{equation}
where $\lambda_1/\mu_1=(n-s+i-3,i-1,j-\ell)/(n-s+i-k-1)$, and $\lambda_2/\mu_2=(n+i-k-j,s-j,s-j)/(i-j+1,i-j-1)$.
To obtain the required number of continuous multiline queues with exactly one wrapping from the second row to the third row, we have to remove the continuous multiline queues with no wrapping from the second row to the third row from the above set. These multiline queues are determined by the following inequalities:
 \[\begin{array}{ccccccccccccc}
     &&  a_{\ell-1}   < & \cdots & <  a_j\\
     &&  \wedge  &  &  \wedge\\
     &b_{1} < &\cdots  <&\cdots & <  b_{i-1}\\
     & \wedge  \ &   &\wedge &  \\
     c_3 <  &\cdots & \cdots< &   c_k
\end{array}
< b_i < 
\begin{array}{ccccccccccccc}
        &a_{j+1}   < & \cdots  < &\cdots  <&\cdots &<a_s\\
       &\wedge & &  &  \wedge\\
       &b_{i+1}  <  &\cdots< & b_s< & b_{s+1}\\
        &\wedge & &  \wedge & & \\
     c_{k+1}  <  &\cdots <& \cdots  < & c_n.
\end{array}\]
The number of these arrangements is 
\begin{equation}\label{prod2}
f_{\lambda_1/\mu_1}.f_{\lambda_3/\mu_3},
\end{equation}
where  $\lambda_3/\mu_3=(n+i-k-j-1,s-j,s-j)/(i-j,i-j-1)$. Summing the difference of two products in~\eqref{prod1} and ~\eqref{prod2} over all possible values of $i,j,k$, and multiplying the sum with $\ell-1$ for each choice of $c_1$, we prove the result.
\end{proof}
\begin{remark}
\label{rem:triplesum}
Unfortunately, we have not been able to derive a closed-form expression for the sum on the right-hand side of \eqref{gammaexpr}. However, based on extensive numerical checks, we have formulated a conjecture for $\gamma^\ell(s,n)$.
\end{remark}
\begin{conj}\label{conj:gamma}
We have, \[\gamma^\ell(s,n)= (\ell-1)\binom{n+2s-\ell}{s-\ell}f_{(n-1,s+1)}. \]
\end{conj}

\noindent When this conjecture is plugged into Lemma~\ref{lem:sum}, it leads to proving the initial condition in \eqref{eq:beta1}. Thus, it would be interesting to find a proof of the conjecture and demonstrate that a tedious-looking triple summation could be simplified into a simple product.

\section*{Acknowledgements}
 
We would like to thank our advisor Arvind Ayyer for all the insightful discussions during the preparation of this paper. We are grateful to Svante Linusson for his fruitful suggestions regarding the proof of Proposition~\ref{prop:incon}. We thank Christoph Koutschan for helping with an earlier approach towards the proof of Theorem~\ref{thm:i<j}.
We also thank the anonymous referees for their valuable revisions and corrections that have significantly improved the quality of the paper.
The second author was supported in part by a SERB grant CRG/2021/001592.

\end{document}